\documentclass[a4paper, DIV=12, 11pt]{scrartcl}
\usepackage[latin1]{inputenc}
\usepackage{amsmath}
\usepackage{amsfonts}
\usepackage{amssymb}
\usepackage{amsthm}
\usepackage{graphicx}
\usepackage{thmtools}
\usepackage{bbm}
\usepackage{hyperref}
\usepackage{enumitem}
\usepackage[bottom, marginal]{footmisc}

\declaretheoremstyle[headfont=\normalfont]{normalhead}
\newtheorem{lemma}{Lemma}[section]
\newtheorem{theorem}[lemma]{Theorem}
\newtheorem{proposition}[lemma]{Proposition}
\newtheorem{corollary}[lemma]{Corollary}
\newtheorem{definition}[lemma]{Definition}

\newtheorem{maintheorem}{Theorem}

\newcommand{\R}{\mathbb{R}}

\newcommand{\Val}{\mathrm{Val}}
\newcommand{\VConv}{\mathrm{VConv}}
\newcommand{\Conv}{\mathrm{Conv}}
\newcommand{\vol}{\mathrm{vol}}
\newcommand{\vsupp}{\mathrm{v}\text{-}{supp }}
\newcommand{\supp}{\mathrm{supp\ }}
\newcommand{\dom}{\mathrm{dom }}
\newcommand{\reg}{\mathrm{reg}}
\newcommand{\epi}{\mathrm{epi}}
\newcommand{\res}{\mathrm{res}}
\newcommand{\diam}{\mathrm{diam}}
\newcommand{\GWVConv}{\overline{\mathrm{GW}}}

\DeclareMathOperator{\GW}{GW}

\setkomafont{title}{\normalfont\Large}

\author{Jonas Knoerr}
\title{The support of dually epi-translation invariant valuations on convex functions}
\date{}

\newcommand{\Addresses}{{
		\bigskip
		\footnotesize
		
		Jonas Knoerr, \textsc{Institut für Mathematik, Goethe-Universität Frankfurt am Main, Robert-Mayer-Str. 10, 60054 Frankfurt, Germany}\par\nopagebreak
		\textit{E-mail address}: \texttt{knoerr@math.uni-frankfurt.de}
		
		\medskip
	}}

\makeatletter
	\def\blfootnote{\xdef\@thefnmark{}\@footnotetext}
\makeatother

\makeindex
\begin{document}
\maketitle
\begin{abstract}
	We study dually epi-translation invariant valuations on cones of convex functions containing the space of finite-valued convex functions. The existence of a homogeneous decomposition is used to associate a distribution to every valuation of this type similar to the Goodey-Weil embedding for translation invariant valuations on convex bodies. The relation between the valuation and its associated distribution is used to establish a notion of support for valuations. As an application, we show that there are no $\mathrm{SL}(n)$ or translation invariant valuations except constant valuations in this class and we discuss which valuations on finite-valued convex functions can be extended to larger cones. In addition, we examine some topological properties of spaces of valuations with compact support.
\end{abstract}

\blfootnote{
	Partially supported by DFG grant BE 2484/5-2.}
\tableofcontents

\section{Introduction}
\subsection{General background}
	Let $V$ be a finite dimensional real vector space and let $\mathcal{K}(V)$ denote the space of convex, compact subsets of $V$. Then a valuation on $\mathcal{K}(V)$ with values in an abelian semigroup $(F,+)$ is a map $\mu:\mathcal{K}(V)\rightarrow F$ that satisfies 
	\begin{align*}
		\mu(K)+\mu(L)=\mu(K\cup L)+\mu(K\cap L)
	\end{align*}
	for all $K,L\in\mathcal{K}(V)$ such that $K\cup L\in \mathcal{K}(V)$. Valuations on convex bodies are a classical part of convex geometry and many structural results have been established (for an overview see for example \cite{Schneider:convex_bodies_Brunn-Minkowski} Chapter 6). 
	For a Hausdorff real topological vector space $F$ let $\Val(V,F)$ denote the space of all valuations $\mu:\mathcal{K}(V)\rightarrow F$ that are translation invariant and continuous with respect to the Hausdorff metric. A valuation $\mu\in\Val(V,F)$ is called \emph{$k$-homogeneous}, or homogeneous of degree $k\in\R$, if $\mu(tK)=t^k\mu(K)$ for all $K\in\mathcal{K}(V)$, $t\ge 0$. Let $\Val_k(V,F)$ denote the subspace of $k$-homogeneous valuations. Then we have the following homogeneous decomposition:
	\begin{theorem}[McMullen \cite{McMullen:Euler_type_McMullen_decomposition}]
		\label{theorem_McMullen_decomposition_for_Val}
		\begin{align*}
			\Val(V,F)=\bigoplus\limits_{i=0}^{\dim V}\Val_i(V,F)
		\end{align*}
	\end{theorem} 
	In other words, $\mu(tK)$ is a polynomial in $t\ge 0$ for $\mu\in\Val(V,F)$ and $K\in\mathcal{K}(V)$, and the degree of this polynomial is bounded by the dimension of $V$. Polarizing this polynomial for a $k$-homogeneous valuation, one obtains the notion of mixed valuations. For example, by polarizing the volume we obtain functionals known as mixed volumes, which play an important role in convex geometry.\\ 
	
	In this paper we are interested in a functional version of this homogeneous decomposition for a special class of valuations on convex functions and additional structures derived from this decomposition. \\
	A valuation on some class of real-valued functions $X$ with values in some abelian semi-group $(F,+)$, is a map $\mu:X\rightarrow F$ satisfying 
	\begin{align*}
		\mu(f)+\mu(h)=\mu(\max(f,h))+\mu(\min(f,h))
	\end{align*}
	for all $f,h\in X$ such that the pointwise maximum $\max(f,h)$ and minimum $\min(f,h)$ belong to $X$. \\
	To see the connection to the notion of valuations on convex bodies, assume that the functions in $X$ are defined on some set $A$. Consider the \emph{epi-graph} $\epi(f):=\{(a,t)\in A\times\R: f(a)\le t\}$ for $f\in X$. Then
	\begin{align*}
		&\epi(\max(f,h))=\epi(f)\cap \epi(h), && \epi(\min(f,h))=\epi(f)\cup \epi(h),
	\end{align*}
	for all $f,h\in X$. Thus a valuation on a space of functions can also be considered as a valuation on their epi-graphs.\\
	
	Although the study of valuations on functions was started only recently, a number of classification results have been established for different spaces of functions, for example Sobolev-spaces \cite{Ludwig:Fisher_information_valuations,Ludwig:valuations_sobolev,Ma:valuations_sobolev}, $\mathrm{L}^p$-spaces  \cite{Ludwig:covariance_matrices_valuations,Ober:Minkowski_valuations_on_Lq_spaces,Tsang:valuations_Lp,Tsang:minkowski_val_Lp}, quasi-concave functions \cite{Bobkov_Colesanti:quermassintegrals_quasi-concave_functions,Colesanti_Lombardi:valuations_quasi-concave,Colesanti_Lombardi_Parapatits:translation_invariant_valuations_quasi-concave}, Orlicz-spaces \cite{Kone:valuations_orlicz}, and functions of bounded variation \cite{Wang:semivaluations_bounded_variations}. Due to their intimate relation to convex bodies, valuations on convex functions have been the object of intense research:
	Monotone or $\mathrm{SL}(n)$-invariant valuations were classified in  \cite{Cavallina_Colesanti:monotone_valuations_convex_functions,Colesanti_Ludwig_Mussnig:minkowski,Colesanti_Ludwig_Mussnig:Valuations_convex_functions,Mussnig:SLn_invariant_super_coercive,Mussnig:volume_polar_volume_euler} and a connection between certain valuations on convex functions  and translation invariant valuations on convex bodies was studied by Alesker in \cite{Alesker:valuations_convex_functions_Monge-Ampere} with the help of Monge-Amp\`ere operators.\\
	
	Let us introduce the general framework for this paper. We will be interested in valuations on convex functions on a finite dimensional real vector space $V$. The spaces of functions under consideration are subspaces of
	\begin{align*}
		\Conv(V):=\{f:V\rightarrow(-\infty,\infty] \ : \ f \text{ convex, lower semi-continuous and proper}\},
	\end{align*}
	where a function $f:V\rightarrow(-\infty,\infty]$ is called \emph{proper} if it is not identical to $+\infty$. This space carries a natural topology, the topology of \emph{epi-convergence}, which we will recall in Section \ref{section_topology_convex_functions}. 
	Especially valuations on the subspaces of finite or super-coercive convex functions 
	\begin{align*}
		\Conv(V,\R)&:=\left\{f\in\Conv(V) \ : \ f<+\infty\right\},\\
		\Conv_{s.c.}(V)&:=\left\{f\in\Conv(V) \ :\ \lim\limits_{x\rightarrow\infty}\frac{f(x)}{|x|}=\infty\right\}
	\end{align*}
	have been of interest, see for example \cite{Colesanti_Ludwig_Mussnig:Hessian_valuations,Colesanti_Ludwig_Mussnig:Valuations_convex_functions,Colesanti_Ludwig_Mussnig:homogeneous_decomposition}. From a valuation theoretic point of view, these two spaces are identical: It was shown in \cite{Colesanti_Ludwig_Mussnig:Hessian_valuations} that the Legendre transform establishes a continuous involution between the two spaces that exchanges minimum and maximum. Thus any continuous valuation on one space induces a unique continuous valuation on the other.\\
	Following the terminology in \cite{Colesanti_Ludwig_Mussnig:homogeneous_decomposition}, let us call a valuation $\mu:\Conv_{s.c.}(V)\rightarrow F$ \emph{translation invariant} if
	\begin{align*}
		\mu(f\circ \tau)=\mu(f) \quad\forall f\in \Conv_{s.c.}(V)
	\end{align*}
	for all translations $\tau:V\rightarrow V$. We will call it \emph{vertically translation invariant} if
	\begin{align*}
		\mu(f+c)=\mu(f) \quad \forall f\in \Conv_{s.c.}(V), c\in\R,
	\end{align*}
	and \emph{epi-translation invariant} if $\mu$ is both translation and vertically translation invariant. These invariance properties are equivalent to the invariance of the valuation with respect to translations of the epigraph of its argument in $V\times\R$. In the dual setting (i.e. considering the composition of an epi-translation invariant valuation with the Legendre transform), this leads to the following notion:\\
	A valuation $\mu:\Conv(V,\R)\rightarrow F$ is called \emph{dually epi-translation invariant}, if
	\begin{align*}
		\mu(f+\lambda+c)=\mu(f) \quad \forall f\in\Conv(V,\R), \lambda\in V^*, c\in\R.
	\end{align*}
	Real-valued valuations on $\Conv(V,\R)$ invariant under the addition of linear functionals were considered by Alesker in \cite{Alesker:valuations_convex_functions_Monge-Ampere}. The examples constructed by him are extensions of functionals of the form
	\begin{align*}
		f\mapsto \int_{\R^n} B(x)\det( H_f(x)[k], A_{1}(x),...,A_{n-k}(x)) dx\quad \text{for }f\in\Conv(\R^n,\R)\cap C^2(\R^n).
	\end{align*}
	Here $H_f$ denotes the usual Hessian of a twice differentiable function $f$, $B\in C_c(\R^n)$, $A_1,...,A_k\in C_c(\R^n,\R^{n\times n})$ are compactly supported functions with values in the space of symmetric $n\times n$-matrices, and $\det$ denotes the mixed determinant of $n$ symmetric matrices ($H_f(x)$ is taken with multiplicity $k$ in the example above). From this representation it is easy to deduce that the valuations Alesker used in his proofs are actually dually epi-translation invariant, so his results also apply to the type of valuations considered in this paper. Similar functionals were also examined in \cite{Colesanti_Ludwig_Mussnig:Hessian_valuations} by Colesanti, Ludwig and Mussnig.\\
	Let us introduce another example: Let $\nu$ be a compactly supported signed Radon measure on $V$ with the property $\int_{V}d\nu=\int_{V}\lambda d\nu=0$ for all $\lambda\in V^*$. Then
	\begin{align*}
		\mu(f):=\int_{V}fd\nu \quad \forall f\in\Conv(V,\R)
	\end{align*}
	defines a dually epi-translation invariant valuation. This class includes valuations such as
	\begin{align*}
		\mu_1(f)&:=f(x)+f(-x)-2f(0) &&\text{for } f\in \Conv(V,\R),\\
		\mu_2(f)&:=\frac{1}{\vol_{n-1}(S^{n-1})}\int_{S^{n-1}} f d\sigma -f(0) && \text{for } f\in\Conv(\R^n,\R),
	\end{align*}
	where $x\in V\setminus \{0\}$ is an arbitrary point and $\sigma$ denotes the usual measure on the unit sphere.\\

\subsection{Results of the present paper}
	We will consider valuations with values in some Hausdorff topological vector space $F$ on certain cones $C\subset\Conv(V)$ containing the space $\Conv(V,\R)$ of finite-valued convex functions. Let $\VConv(C;V,F)$ denote the space of continuous, dually epi-translation invariant valuations on $C$ with values in $F$ (see Section \ref{section_valuations_convex_functions} for the precise definition). A valuation $\mu\in\VConv(C;V,F)$ will be called \emph{$k$-homogeneous} if $\mu(tf)=t^k\mu(f)$ for all $f\in C$ and $t> 0$, and the subspace of $k$-homogeneous valuations will be denoted by $\VConv_k(C;V,F)$. Then we have the following decomposition:
	\begin{maintheorem}
		\label{maintheorem_homogeneous_decomposition}
		If $C\subset\Conv(V)$ is a cone containing $\Conv(V,\R)$, then
		\begin{align*}
		\VConv(C;V,F)=\bigoplus\limits_{i=0}^{\dim V}\VConv_i(C;V,F).
		\end{align*}
	\end{maintheorem}
	In \cite{Colesanti_Ludwig_Mussnig:homogeneous_decomposition}, this was proved by Colesanti, Ludwig and Mussnig for real-valued valuations on $\Conv(V,\R)$. The more general statement can be deduced from their proof with only minor modifications. Instead of repeating their arguments, we will present a slightly different proof in Section \ref{section_homogeneous_decomposition} using an embedding of $\VConv(C;V,F)$ into $\Val(V^*\times\R,F)$ using support functionals of convex bodies. \\
	
	Similar to the classical McMullen decomposition (and to \cite{Colesanti_Ludwig_Mussnig:homogeneous_decomposition} Theorem 23), this allows us to define the polarization of any homogeneous valuation $\mu\in \VConv_k(C;V,F)$ under the following regularity assumption on the cone: We will call a cone $C\subset\Conv(V)$ \emph{regular} if every function $f\in C$ satisfies $f<+\infty$ on some open subset in $V$. We obtain a symmetric functional $\bar{\mu}:C^k\rightarrow F$ which satisfies $\bar{\mu}(f,...,f)=\mu(f)$ for all $f\in C$ and which is a continuous, $1$-homogeneous and \emph{additive} valuation in each coordinate, i.e. it satisfies
	\begin{align*}
		\bar{\mu}(f+h,f_2,\dots,f_k)=\bar{\mu}(f,f_2,\dots,f_k)+\bar{\mu}(h,f_2,\dots,f_k)
	\end{align*} 
	for all $f,h, f_2,...,f_k\in C$. Goodey and Weil used the polarization of a valuation in $\Val_k(\R^n)$ to define a distribution on the $k$-fold product of the unit sphere in $\R^n$ (see \cite{Goodey_Weil:Distributions_and_valuations}). Alesker proved in \cite{Alesker:McMullenconjecture} that the support of this distribution is contained in the diagonal of this Cartesian product of spheres, which plays a crucial role in the proof of his Irreducibility Theorem (see \cite{Alesker:McMullenconjecture,Alesker:IrreducibilityThm}).\\
	Following the ideas of Goodey-Weil and Alesker, we establish the following version of the Goodey-Weil embedding for valuations on convex functions in Section \ref{section_Goodey-Weil-embedding}:
	\begin{maintheorem}
		\label{maintheorem_GW}
		Let $F$ be a locally convex vector space admitting a continuous norm, $C\subset\Conv(V)$ a regular cone containing $\Conv(V,\R)$. For every $\mu\in\VConv_k(C;V,F)$ there exists a unique distribution $\GWVConv(\mu)\in\mathcal{D}'(V^k,\bar{F})$ with compact support which satisfies the following property: If $f_1,...,f_k\in \Conv(V,\R)\cap C^\infty(V)$, then
		\begin{align*}
			\GWVConv(\mu)(f_1\otimes...\otimes f_k)=\bar{\mu}(f_1,...,f_k),
		\end{align*}
		where $\bar{\mu}$ is the polarization of $\mu$.\\
		Moreover, the support of $\GWVConv(\mu)$ is contained in the diagonal in $V^k$.
	\end{maintheorem}
	Here $\bar{F}$ denotes the completion of $F$ and $\mathcal{D}'(V^k,\bar{F})$ denotes the space of all distributions on $V^k$ with values in $\bar{F}$, i.e. the space of all continuous functionals $\phi:C^\infty_c(V^k)\rightarrow \bar{F}$, where $C^\infty_c(V^k)$ is equipped with the usual topology on test functions. We will call the distribution $\GWVConv(\mu)$ the Goodey-Weil distribution of $\mu$. We thus obtain an injective map $\GWVConv:\VConv_k(C;V,F)\rightarrow \mathcal{D}'(V^k,\bar{F})$, which will be called the Goodey-Weil embedding.\\

	More generally, we also define a version of the Goodey-Weil distribution for valuations with values in an arbitrary locally convex vector space. Although the Goodey-Weil distribution is still uniquely determined by its underlying valuation in this case, the support is, in general, not compact. To illustrate this fact, we examine the Hessian measures (see \cite{Colesanti_Ludwig_Mussnig:Hessian_valuations}), which can be considered as continuous, dually epi-translation invariant valuations with values in the space of signed Radon measures on $V$ (equipped with the vague topology). This example was also our main motivation to examine valuations with values in arbitrary locally convex vector spaces.\\
		
	In Section \ref{section_support_of_a_valuation} we use the support of the Goodey-Weil distribution to obtain a notion of support for the corresponding valuation and we discuss how the support can be defined intrinsically. As an application, we show that there are no non-trivial, real-valued dually epi-translation invariant valuations (except constant valuations) that are also invariant under the operation of the special linear group or translations. We also characterize the image of the embedding used in the proof of Theorem \ref{maintheorem_homogeneous_decomposition}.\\
		 \\
	It is natural to consider the subspace $\VConv_A(C;V,F)$ of valuations with support contained in a fixed subset $A\subset V$. In Section \ref{section_subspaces_compact_support} we discuss some topological properties of $\VConv_A(V,F)$ in the case where $A$ is compact and we show the following topological result:
	\begin{proposition}
		\label{mainproposition_approximation_using_sequences}
		Let $F$ be a locally convex vector space admitting a continuous norm.
		If a sequence $(\mu_j)_j$ converges to $\mu$ in $\VConv(V,F)$, then there exists a compact set $A\subset V$ such that the support of $\mu$ and $\mu_j$ is contained in $A$ for all $j\in\mathbb{N}$. In particular, $(\mu_j)_j$ converges to $\mu$ in $\VConv_A(V,F)$.
	\end{proposition}
		
		In Section \ref{section_description_valuations_on_cones_with_support} we relate the valuations in $\VConv(C;V,F)$ to the common domains of $f\in C$, i.e. the sets $\dom f:=\{x\in V \ : \  f(x)<\infty\}$. This also gives a partial answer to the question, which valuations on finite-valued convex functions can be extended to larger cones in $\Conv(V)$. Together with Corollary \ref{corollary_existence_non_trivial_valuations}, this refines \cite{Colesanti_Ludwig_Mussnig:homogeneous_decomposition} Theorem 30, where valuations on the maximal cone $C=\Conv(V)$ were considered.\\
		
		These results apply in particular to cones of the form $C_U:=\{f\in\Conv(V) \ : \  f|_U<\infty\}$ for some open and convex subset $U\subset V$. In Section \ref{section_valuations_on_open_sets}, we considering valuations on the space $\Conv(U,\R):=\{f:U\rightarrow\mathbb{R} \ : \  f\text{ convex }\}$ with values in a locally convex vector space $F$ that are dually epi-translation invariant and continuous with respect to the topology of uniform convergence on compact subsets in $U$. Denoting the space of these functionals by $\VConv(U,F)$, we can identify this space of valuations with valuations on the cone $C_U$:
			\begin{maintheorem}
				\label{maintheorem_isomorphism_cone_open_subset}
				If $U\subset V$ is an open convex subset and $F$ is a locally convex vector space admitting a continuous norm, then the map
				\begin{align*}
				\res^*:\VConv(U,F)&\rightarrow \VConv(C_U;V,F)\\
				\mu&\mapsto[f\mapsto \mu(f|_U)]
				\end{align*}
				is a topological isomorphism.
			\end{maintheorem}
		Here both spaces are equipped with the topology of uniform convergence on compact subsets in $\Conv(U,\R)$ and $C_U$ respectively.
		
		\paragraph{Acknowledgments}
		Part of this paper was written during a stay at the Universit\`a degli Studi di Firenze and I want to thank the university and especially Andrea Colesanti for the hospitality. I also want to thank Andreas Bernig for many useful discussions, suggestions and encouragement during this project, as well as his comments and remarks on the first draft of this paper.

\section{Convex functions}
	\label{section_convex_functions}
	In this section, we collect some facts about the space of convex functions and its topology. For simplicity we will assume that $V$ is a euclidean vector space. Let $B_R:=B_R(0)$ denote the closed ball of radius $R>0$ in $V$.\\
	
	Let $U\subset V$ be a convex subset. A function $f:U\rightarrow\R\cup\{+\infty\}$ is called convex if for any $x_0,x_1\in U$ and every $t\in[0,1]$ the following inequality holds: 
	\begin{align*}
		f(tx_0+(1-t)x_1)\le tf(x_0)+(1-t)f(x_1).
	\end{align*}
	Equivalently, $f:U\rightarrow\R\cup\{+\infty\}$ is convex if and only if its \emph{epi-graph} $\mathrm{epi}(f):=\{(x,t)\in U\times\R \ : \  f(x)\le t\}$ is a convex subset in $U\times\R$. Note that $f$ is lower semi-continuous if and only if $\epi(f)$ is closed in $U\times\R$. In this paper we are mostly interested in subsets of
	\begin{align*}
		\Conv(V):=\{f:V\rightarrow\R\cup\{+\infty\} \ : \  f\text{ convex, lower semi-continuous}, f\not\equiv +\infty\}.
	\end{align*}
	For any $f\in\Conv(V)$, we define the domain of $f$
	\begin{align*}
		\mathrm{dom}(f):=\{x\in V \ : \ f(x)<+\infty\}.
	\end{align*}
	By definition, $\mathrm{dom}(f)$ is a non-empty convex subset of $V$. $f$ is always continuous on the interior of $\mathrm{dom}(f)$. In particular, the space of finite-valued convex functions
	\begin{align*}
		\Conv(V,\R):=\{f:V\rightarrow\R \ : \ f \text{ convex}\}\subset\Conv(V)
	\end{align*}
	contains only continuous functions. Note that $\Conv(V,\R)$ is closed under the formation of the pointwise maximum, while the maximum of two elements of $\Conv(V)$ may be identical to $+\infty$. 
		
	\subsection{Topology on the space of convex functions}
		\label{section_topology_convex_functions}
		\begin{definition}
			A sequence $(f_j)_j$ in $\Conv(V)$ \emph{epi-converges} to $f\in\Conv(V)$ if and only if for every $x\in V$ the following conditions hold:
			\begin{enumerate}
				\item For every sequence $(x_j)_j$ in $V$ converging to $x$: $f(x)\le \liminf\limits_{j\rightarrow\infty}f_j(x_j)$.
				\item There exists a sequence $(x_j)_j$ converging to $x$ such that $f(x)= \lim\limits_{j\rightarrow\infty}f_j(x_j)$.
			\end{enumerate}
		\end{definition}
		It is known that this notion of convergence is induced by a metrizable topology on $\Conv(V)$ (see for example \cite{Rockafellar_Wets:Variational_analysis} Theorem 7.58).\\
		In the constructions used in this paper, the limit function $f\in\Conv(V)$ will usually be finite on some open subset of $V$. In this case epi-convergence, pointwise convergence and locally uniform convergence are compatible in the following sense:
		\begin{proposition}\cite[Theorem 7.17]{Rockafellar_Wets:Variational_analysis}
			\label{proposition_convergence_finite_convex_functions} 
			For a function $f\in\Conv(V)$ such that $\dom f$ has non-empty interior and a sequence $(f_j)_j$ in $\Conv(V)$ the following are equivalent:
			\begin{enumerate}
				\item $(f_j)_j$ epi-converges to $f$.
				\item $(f_j)_j$ converges pointwise to $f$ on a dense subset.
				\item $(f_j)_j$ converges uniformly to $f$ on all compact subsets that do not contain a boundary point of $\dom f$. 
			\end{enumerate}
		\end{proposition}
		In particular, a sequence $(f_j)_j$ in $\Conv(V,\R)$ epi-converges to $f\in\Conv(V,\R)$ if and only if it converges uniformly on compact subsets.
	\subsection{Compactness in $\Conv(V,\R)$}
	\begin{proposition}
		\label{proposition_convex_functions_local_lipschitz_constants}
		Let $U\subset V$ be a convex open subset and $f:U\rightarrow\R$ a convex function. If $X\subset U$ is a set with $X+\epsilon B_1\subset U$ and $f$ is bounded on $X+ B_\epsilon$, then $f$ is Lipschitz continuous on $X$ with Lipschitz constant $\frac{2}{\epsilon}\|f|_{X+ B_\epsilon}\|_\infty$.
	\end{proposition}
	\begin{proof}
		This is a special case of \cite{Rockafellar_Wets:Variational_analysis} 9.14.
	\end{proof}
	\begin{proposition}
		\label{proposition_compactness_Conv}
		A subset $U\subset \Conv(V,\R)$ is relatively compact if and only if it is bounded on compact subsets.
	\end{proposition}
	\begin{proof}
		As the topology on $\Conv(V,\R)$ is metrizable, we only have to show that the closure of such a subset is sequentially compact. Let $(f_k)_k$ be a sequence in $\Conv(V,\R)$ that is bounded on compact subsets of $V$. Then the Lipschitz constants of these functions are also uniformly bounded on $B_j$ for all $j\in\mathbb{N}$ by Proposition \ref{proposition_convex_functions_local_lipschitz_constants}. In particular, the set $\{f_k|_{B_j} \ : \ k\in\mathbb{N}\}\subset C(B_j)$ is equicontinuous. By the theorem of Arzel\`a-Ascoli, we can choose a subsequence $f_{j,k}$ that converges uniformly on $B_j$ to some function $f_{j,\infty}\in C(B_j)$ for $k\rightarrow\infty$. Iterating this argument for all $j\in\mathbb{N}$ and taking an appropriate diagonal series, we find a subsequence that converges uniformly on $B_j$ for all $j\in\mathbb{N}$ to some function $f\in C(V)$. It is easy to see that $f$ is convex. Now the claim follows from Proposition \ref{proposition_convergence_finite_convex_functions}.
	\end{proof}

	\subsection{The Legendre transform, subdifferentials and some density results}
		The \emph{Legendre transform} or \emph{convex dual} of a function $f\in\Conv(V)$ is the function $f^*:V^*\rightarrow (-\infty,\infty]$ given by
		\begin{align*}
		f^*(y)=\sup\limits_{x\in V}\langle y,x\rangle -f(x) \quad \text{for }y\in V^*,
		\end{align*} 
		where $\langle \cdot,\cdot\rangle:V^*\times V\rightarrow\R$ denotes the canonical pairing.
		As a consequence of \cite{Rockafellar:Convex_Analysis} Theorem 12.2 and Corollary 12.2.1, we have
		\begin{proposition}
			For $f\in\Conv(V)$, $f^*\in\Conv(V^*)$ and $f^{**}:=(f^*)^*=f$.
		\end{proposition}
		Let $f\in\Conv(V)$. An element $y\in V^*$ is called a \emph{subgradient} of $f$ in $x_0\in V$ if 
		\begin{align*}
			f(x_0)+\langle y,x-x_0\rangle\le f(x) \quad \text{for all }x\in V.
		\end{align*}
		The set of all subgradients of $f$ in a point $x_0\in V$ is called the subdifferential of $f$ in $x_0$ and will be denoted by $\partial f(x_0)$.\\
		We recall the following basic properties of the subdifferential:
		\begin{lemma}[\cite{Rockafellar:Convex_Analysis} Theorem 23.5]
			\label{lemma_properties_subgradients}
			For $f\in\Conv(V)$, $x\in V$, $y\in V^*$ the following are equivalent:
			\begin{enumerate}
				\item $y\in\partial f(x)$,
				\item $\langle y,x\rangle =f(x)+f^*(y)$,
				\item $y\in \mathrm{argmax}_{x\in V}\langle y,x\rangle -f(x)$.
			\end{enumerate}
		\end{lemma}
		\begin{lemma}
			\label{lemma_compactness_face_epigraph_conjugate-function}
			Let $f\in\Conv(V,\R)$. Then $\partial f(x)\in\mathcal{K}(V^*)$ for all $x\in V$ and if $f$ is $L$ Lipschitz continuous on a neighborhood of $x\in V$, then $\partial f(x)\subset B_L(0)$. 
		\end{lemma}
		\begin{proof}
			This is a special case of \cite{Clarke:optimization_non_smooth_analysis} 2.1.2.
		\end{proof}
		Also note that $\partial f(x)\ne \emptyset$ for any $x\in V$, if $f\in\Conv(V,\R)$.\\
		The support function $h_K\in\Conv(V,\R)$ of a convex body $K\in\mathcal{K}(V^*)$ is defined as
		\begin{align*}
			h_K(y):=\sup\limits_{x\in K}\langle y,x\rangle.
		\end{align*}
		These functions will play an important role in Section \ref{section_relation_to_val_convex_bodies}, where we relate valuations on convex functions to valuations on convex bodies. This construction relies on a density result contained in Corollary \ref{corollary_density_various_families_of_convex_functions} below, which we will deduce from the following proposition.
		\begin{proposition}
			\label{proposition_epi_graph_support_function_convex_body}
			Let $f\in\Conv(V,\R)$ be a finite-valued convex function, $f^*$ its convex dual, $R>0$. If $\|f\|_{C(B_{R+2})}\le c$, then the set 
			\begin{equation*}
			K_{f,R}:=\epi(f^*)\cap \{(y,t)\in V^*\times\R  \ : \ |y|\le 2c,|t|\le (2R+3)c\}
			\end{equation*} is a convex body in $V^*\times\R$ and satisfies
			\begin{equation*}
			f(x)=h_{K_{f,R}}(x,-1) \quad\text{for all }x\in B_{R+1}.
			\end{equation*}
		\end{proposition}
		\begin{proof}
			Consider the set $C:=\{y\in V^* \ : \ y\in\partial f(x) \text{ for some }x\in B_{R+1}\}$. As $f|_{B_{R+1}}$ is Lipschitz continuous with Lipschitz constant $L=2\|f\|_{C(B_{R+2})}$ by Proposition \ref{proposition_convex_functions_local_lipschitz_constants}, Lemma \ref{lemma_compactness_face_epigraph_conjugate-function} implies that $C$ is contained in a ball of radius $L$ centered at the origin. \\
			Any $y\in C$ satisfies $f^*(y)=\langle x,y\rangle -f(x)$ for some $x\in B_{R+1}$ due to Lemma \ref{lemma_properties_subgradients}. Thus
			\begin{equation*}
			|f^*(y)|\le |\langle y,x\rangle| +|f(x)|\le L(R+1)+\|f\|_{C(B_{R+2})}\le (2R+3)\|f\|_{C(B_{R+2})}\le (2R+3)c.
			\end{equation*}
			Let us show that $f(x)=h_{K_{f,R}}(x,-1)$ for all $x\in B_{R+1}$.	Obviously, the left hand side is equal to or larger than the right hand side. By Lemma \ref{lemma_compactness_face_epigraph_conjugate-function}, we know that for any $x\in B_{R+1}$ there exists $y\in V^*$ such that $f(x)=\langle y,x\rangle-f^*(y)$. In particular $y\in C$. Then 
			\begin{align*}
			(y,f^*(y))\in \epi(f^*)\cap \{(y,t)\in V^*\times\R  \ : \ |y|\le 2c,|t|\le (2R+3)c\}=K_{f,R}
			\end{align*}
			by the previous discussion, so
			\begin{align*}
			f(x)=\langle y,x\rangle-f^*(y)\le \sup_{(\tilde{y},t)\in K_{f,R}}\langle \tilde{y},x\rangle -t= h_{K_{f,R}}(x,-1).
			\end{align*}
			As $f^*$ is lower semi-continuous, the set $\epi(f^*)\cap \{(y,t)\in V^*\times\R  \ : \ |y|\le 2(R+2)c,|t|\le 3(R+2)c\}$ is closed. As it is also convex and bounded, it  belongs to $\mathcal{K}(V^*\times\R)$.
		\end{proof}
		We thus obtain the following density results.
		\begin{corollary}
			\label{corollary_density_various_families_of_convex_functions}
			The following families of functions are dense in $\Conv(V,\R)$:
			\begin{enumerate}
				\item $\{h_K(\cdot,-1) \ : \ K\in\mathcal{K}(V^*\times\R)\}$,
				\item $\{h_P(\cdot,-1) \ : \ P\in\mathcal{K}(V^*\times\R) \text{ polytope}\}$,
				\item $\{h_K(\cdot,-1) \ : \ K\in\mathcal{K}(V^*\times\R)\text{ smooth and strictly convex}\}$,
				\item $\Conv(V,\R)\cap C^\infty(V)$.
			\end{enumerate}
		\end{corollary}
		\begin{proof}
			For the first set this follows directly from Proposition \ref{proposition_epi_graph_support_function_convex_body} and the continuity of the map $\mathcal{K}(V^*\times\R)\rightarrow\Conv(V,\R)$, $K\mapsto h_K(\cdot,-1)$, see Lemma \ref{lemma_continuity_support_function} below. As $\{P\in\mathcal{K}(V^*\times\R) \ : \  P \text{ polytope}\}$ and $\mathcal{K}(V^*\times\R)^{sm}$
			are dense subsets of $\mathcal{K}(V^*\times\R)$ respectively, this implies the density of the second and third set. For the last set, observe that the support function of any smooth and strictly convex body is smooth, so the last set contains a dense subset and is thus dense itself.
		\end{proof}
		
	\subsection{Lipschitz regularization}
	Most of our results are actually results on valuations on $\Conv(V,\R)$ which generalize to more general subspaces of $\Conv(V)$ by approximation.\\
	For $r>0$ the \emph{Lipschitz regularization} or \emph{Pasch-Hausdorff envelope} of a convex function $f\in\Conv(V)$ is defined as
	\begin{align*}
	\reg_r(f):=\left(f^*+1^\infty_{B_{1/r}}\right)^*.
	\end{align*} 
	We will need the following properties:
	\begin{proposition}[\cite{Colesanti_Ludwig_Mussnig:Hessian_valuations} Propositions 4.1, 4.2, 4.3]
		\label{proposition_properties_LIpschitz_regularization}
		For $f,h\in\Conv(V)$ and $r>0$, the Lipschitz regularization has the following properties:
		\begin{enumerate}[label=\roman*.]
			\item There exists $r_0>0$ such that $\reg_rf\in\Conv(V,\R)$ for all $0<r\le r_0$.
			\item $\reg_r f$ epi-converges to $f$ for $r\rightarrow 0$.
			\item If $x\in\dom(f)$ and $\partial f(x)\cap B_{1/r}\ne\emptyset$, then $\reg_r f(x)=f(x)$ and $\partial \reg_r f(x)=\partial f(x)\cap B_{1/r}$.
			\item If $(f_j)_j$ is a sequence in $\Conv(V)$ that epi-converges to $f$, then there exists $r_0>0$ such that $(\reg_rf_j)_j$ epi-converges to $\reg_rf$ for all $0<r\le r_0$.
			\item If $\min(f,h)$ is convex, then there exists $r_0>0$ such that 
			\begin{align*}
			&\reg_r(\max(f,h))=\max(\reg_r f,\reg_r h) && \reg_r(\min(f,h))=\min(\reg_r f,\reg_r h)
			\end{align*} 
			for all $0<r\le r_0$.
		\end{enumerate}
	\end{proposition}	
	Note that ii. implies that $\Conv(V,\R)\subset \Conv(V)$ is dense. Thus the sets considered in Corollary \ref{corollary_density_various_families_of_convex_functions} are also dense in $\Conv(V)$.

\section{Dually epi-translation invariant valuations on convex functions}
	\label{section_valuations_convex_functions}	
		Let $C\subset\Conv(V)$ be a non-empty subset and let $(F,+)$ be an abelian semigroup.
		\begin{definition}
			A map $\mu:C\rightarrow F$ is called a valuation if 
			\begin{align*}
				\mu(f)+\mu(h)=\mu(\max(f,h))+\mu(\min(f,h))
			\end{align*}
			for all $f,h\in C$ such that the pointwise maximum $\max(f,h)$ and minimum $\min(f,h)$ belong to $C$.
		\end{definition}
		A valuation $\mu:C\rightarrow F$ is called \emph{dually epi-translation invariant} if
		\begin{align*}
			\mu(f+\lambda+c)=\mu(f)\quad \forall f\in C, \lambda\in V^*, c\in\R \text{ s.t. } f+\lambda+c\in C.
		\end{align*}
		\begin{definition}
			If $F$ is a topological vector space, we let $\VConv(C;V,F)$ denote the space of all valuations $\mu:C\rightarrow F$ that are
			\begin{enumerate}
				\item continuous with respect to epi-convergence,
				\item dually epi-translation invariant.
			\end{enumerate}
		\end{definition}
		If $C=\Conv(V,\R)$, we will use the notation $\VConv(V,F)$ instead. For $F=\R$, we will write $\VConv(C;V):=\VConv(C;V,\R)$ and $\VConv(V):=\VConv(V,\R)$ for brevity.\\
			
		We also equip $\VConv(C;V,F)$ with the compact-open topology, which is generated by the open sets 
		\begin{align*}
			\mathcal{M}(K,O):={\{\mu\in\VConv(C;V,F) \ : \ \mu(f)\in O\quad\forall f\in K\}}
		\end{align*}
		for $K\subset C$ compact and $O\subset F$ open. If $F$ is locally convex, $\VConv(C;V,F)$ is a locally convex vector space equipped with the family of semi-norms
		\begin{align*}
			\|\mu\|_{F;K}:=\sup_{f\in K}|\mu(f)|_F,
		\end{align*}
		where $K\subset C$ is compact and $|\cdot|_F$ is a continuous semi-norm on $F$. It is easy to see that $\VConv(C;V,F)$ is complete if $F$ is complete. Also note that Proposition \ref{proposition_compactness_Conv} provides a characterization of all compact subsets for $C=\Conv(V,\R)$.
	\subsection{Relation to valuations on convex bodies}
		\label{section_relation_to_val_convex_bodies}
		As noted in Section \ref{section_convex_functions}, the support function $h_K$ of $K\in\mathcal{K}(V^*)$ is defined by $h_K(y):=\sup_{x\in K}\langle y,x\rangle$. It has the following well known properties:
	\begin{enumerate}
		\item $h_K\in\Conv(V^*,\R)$.
		\item $h_{tK}=th_K$ for all $t\ge 0$.
		\item If $K,L$ are convex bodies such that $K\cup L$ is convex, then $h_{K\cup L}=\max(h_K, h_L)$ and  $h_{K\cap L}=\min(h_K, h_L)$.
		\item A sequence  $(K_j)_j$ of convex bodies converges to $K$ with respect to the Hausdorff metric if and only if $(h_{K_j})_j$ converges to $h_K$ uniformly on compact subsets.
	\end{enumerate}
		The last property implies
	\begin{lemma}
		\label{lemma_continuity_support_function}
		The map
		\begin{align*}
		P:\mathcal{K}(V^*\times\R)&\rightarrow\Conv(V,\R)\\
		K&\mapsto h_K(\cdot,-1)
		\end{align*}
		is continuous.
		\end{lemma}
		Here we have used the canonical isomorphism $(V\times\R)^*\cong V^*\times\R$. For $\mu\in\VConv(C;V,F)$ we define $T(\mu):\mathcal{K}(V^*\times\R)\rightarrow F$ by $T(\mu)[K]:=\mu(h_K(\cdot,-1))$.
		\begin{theorem}
			\label{theorem_embedding VConv->Val(VxR)}
			Let $F$ be a Hausdorff real topological vector space and $C\subset\Conv(V)$ a subset with $\Conv(V,\R)\subset C$. Then 
			\begin{equation*}
				T:\VConv(C;V,F)\rightarrow \Val(V^*\times\R,F)
			\end{equation*} is well defined, continuous, and injective.
		\end{theorem}
			\begin{proof}
				It is clear that $T(\mu)=\mu\circ P\in\Val(V^*\times\R,F)$.\\
				Let us show that $T$ is injective: If $T(\mu)=0$, then $\mu(h_K(\cdot,-1))=0$ for all $K\in\mathcal{K}(V^*\times\R)$. By Corollary \ref{corollary_density_various_families_of_convex_functions}, these functions form a dense subspace of $\Conv(V,\R)$, which is dense in $C$, so the continuity of $\mu$ implies $\mu=0$, as $F$ is Hausdorff. Thus $T$ is injective.\\
				A basis for the topology of $\Val(V^*\times\R,F)$ is given by the open sets 
				\begin{align*}
					\mathcal{M}(B,O)=\{\mu\in\Val(V^*\times\R,F) \ : \ \mu(K)\in O\quad\forall K\in B\},
				\end{align*} 
				where $O\subset F$ is open and $B\subset\mathcal{K}(V^*\times\R)$ is compact. Then
				\begin{align*}
					T^{-1}(\mathcal{M}(B,O))&=\{\mu\in\VConv(C;V,F) \ : \ \mu(h_K(\cdot,-1))\in O\quad\forall K\in B\}\\
					&=\{\mu\in\VConv(C;V,F) \ : \ \mu(f)\in O\quad\forall f\in P(B)\}\\
					&=\mathcal{M}(P(B),O).
				\end{align*}
				As $P$ is continuous, $P(B)$ is compact in $C$, so $T^{-1}(\mathcal{M}(B,O))=\mathcal{M}(P(B),O)$ is open in $\VConv(C;V,F)$.
			\end{proof}

\section{Homogeneous decomposition}
	\label{section_homogeneous_decomposition}
	\subsection{Proof of Theorem 1}
		In this section we prove Theorem \ref{maintheorem_homogeneous_decomposition}.
		Let $F$ be a Hausdorff real topological vector space and $C\subset\Conv(V)$. 
		\begin{definition}
			A continuous valuation $\mu:C\rightarrow F$ is called $k$-homogeneous if $\mu(tf)=t^k\mu(f)$ for all $f\in C$ and for all $t>0$ such that $tf\in C$. We will denote the space  of $k$-homogeneous valuations in $\VConv(C;V,F)$ by $\VConv_k(C;V,F)$.
		\end{definition}
		We will call a subset $C\subset\Conv(V)$ a cone, if $f+th\in C$ for all $f,h\in C$, $t>0$.
		\begin{proposition}
			\label{proposition_McMullen_for_Conv(V)_with_n+1_degree}
			Let $F$ be a Hausdorff real topological vector space, $C\subset\Conv(V)$ a cone containing $\Conv(V,\R)$, and $\mu\in \VConv(C;V,F)$. Then there exist valuations $\mu_i\in\VConv_i(C;V,F)$, $i=0,...,n+1$ such that
			\begin{equation*}
				\mu=\sum_{i=0}^{n+1}\mu_i.
			\end{equation*}
		\end{proposition}
		\begin{proof}
			Consider the injective map $T:\VConv(C;V,F)\rightarrow\Val(V^*\times\R,F)$ from Theorem \ref{theorem_embedding VConv->Val(VxR)}, given by $T(\mu)[K]=\mu(h_K(\cdot,-1))$ for $K\in\mathcal{K}(V^*\times\R)$.\\
			For $t>0$ define $\mu^t\in\VConv(C;V,F)$ by $\mu^t(f):=\mu(t f)$ for $f\in C$. Then $T(\mu^t)[K]=T(\mu)[tK]$, as $h_{tK}=th_K$ for $t>0$.\\
			Using the McMullen decomposition (Theorem \ref{theorem_McMullen_decomposition_for_Val}) for $\Val(V^*\times\R,F)$, we see that $T(\mu^t)=\sum_{i=0}^{n+1}t^i\tilde{\mu}_i$ for homogeneous elements $\tilde{\mu}_i\in \Val_i(V^*\times\R,F)$. Plugging in $0<t_0<\dots<t_{n+1}$ and using the inverse of the Vandermonde matrix, we obtain constants $c_{ij}\in \R$ such that $\tilde{\mu}_i=\sum_{j=0}^{n+1}c_{ij}T(\mu^{t_j})$.\\
			Now define $\mu_i\in\VConv(C;V,F)$ by $\mu_i:=\sum_{j=0}^{n+1}c_{ij}\mu^{t_j}$. Then, obviously, $T(\mu_i)=\tilde{\mu}_i$ and for any $K\in\mathcal{K}(V^*\times\R)$, $t>0$:
			\begin{align*}
			T\left(\mu_i^t\right)(K)=T(\mu_i)(tK)=\tilde{\mu}_i(tK)=t^i\tilde{\mu}_i(K)=t^iT(\mu_i)(K)=T\left(t^i\mu_i\right)(K).
			\end{align*}
			The injectivity of $T$ implies $t^i\mu_i=\mu_i^t$, i.e. $\mu_i$ is $i$-homogeneous. In addition,
			\begin{align*}
			T(\mu)=\sum_{i=0}^{n+1}\tilde{\mu}_i=\sum_{i=0}^{n+1}T(\mu_i)=T\left(\sum_{i=0}^{n+1}\mu_i\right).
			\end{align*}
			Thus the injectivity of $T$ implies $\mu=\sum_{i=0}^{n+1}\mu_i$.
		\end{proof}
		To show that the $n+1$-homogeneous component is trivial, we need the following lemma:
		\begin{lemma}
			\label{lemma_injectivity_characteristic_function}
			Let $(G,+)$ be an Abelian semi-group with cancellation law that carries a Hausdorff topology, and $\mu_1,\mu_2:\Conv(V,\R)\rightarrow G$ two continuous valuations. If $\mu_1(h_P(\cdot-y)+c)=\mu_2(h_P(\cdot-y)+c)$ for all polytopes $P\in\mathcal{K}(V^*)$ with $0\in \mathrm{int} P$, $y\in V$ and $c\in\R$, then $\mu_1\equiv \mu_2$ on $\Conv(V,\R)$.
		\end{lemma}
		\begin{proof}
			This is \cite{Mussnig:volume_polar_volume_euler} Lemma 5.1. To be precise, the version in \cite{Mussnig:volume_polar_volume_euler} considers translation invariant valuations, however, the proof only uses the weaker property stated above.
		\end{proof}
		\begin{proposition}
			\label{proposition_n+1-degree=0}
			$\VConv_{n+1}(C;V,F)=0$
		\end{proposition}
		\begin{proof}
			Let $\mu\in\VConv_{n+1}(C;V,F)$. As $\Conv(V,\R)$ is dense in $C$, we only need to show that $\mu$ vanishes on finite-valued convex functions. Using Lemma \ref{lemma_injectivity_characteristic_function}, it is thus sufficient to show $\mu(h_K(\cdot-y))=0$ for all $K\in\mathcal{K}(V^*)$ and $y\in V$. However, $K\mapsto \mu(h_K(\cdot-y))$ defines an element of $\Val_{n+1}(V^*,F)=0$. The claim follows.
		\end{proof}
		\begin{proof}[Proof of Theorem \ref{maintheorem_homogeneous_decomposition}]
			This follows directly from Proposition \ref{proposition_McMullen_for_Conv(V)_with_n+1_degree} and Proposition \ref{proposition_n+1-degree=0}.
		\end{proof}
	\subsection{Polynomiality and polarization}
		From now on we will assume that $F$ is a Hausdorff real topological vector space and that $C\subset\Conv(V)$ is a cone containing $\Conv(V,\R)$. Then we can consider the question of polynomiality for elements of $\VConv(C;V,F)$. 
		To define the polarization of a homogeneous valuation in $\VConv_k(C;V,F)$, we need the following regularity assumption on the cone $C$:
		\begin{definition}
		 	\label{definition:valuations-on-conv:regular-cone}
		 	A cone $C\subset\Conv(V)$ will be called \emph{regular} if $\dom f$ has non-empty interior for all $f\in C$.
		\end{definition} 
			\begin{lemma}
				\label{lemma:valuations-on-conv:continuity-addition}
				If $C\subset \Conv(V)$ is a regular cone, then 
				\begin{align*}
				+:C\times C&\rightarrow C\\
				(f,h)&\mapsto f+h
				\end{align*}
				is continuous.
			\end{lemma}
			\begin{proof}
				This follows directly from Proposition \ref{proposition_convergence_finite_convex_functions}.
			\end{proof}
			Note that this map is in general not continuous if $C$ is not a regular cone. From Theorem \ref{maintheorem_homogeneous_decomposition} we deduce
		\begin{corollary}
			\label{corollary_polynomiality_valuations}
			Let $C\subset\Conv(V)$ be a regular cone with $\Conv(V,\R)\subset C$. For $\mu\in \VConv(C;V,F)$, and $f_1,...,f_m\in C$, $\mu(\sum_{j=1}^{m}\lambda_j f_j)$ is a polynomial in $\lambda_j> 0$.
		\end{corollary}
		\begin{proof}
			We will use induction on $m\in\mathbb{N}$. For $m=1$, this is just Theorem \ref{maintheorem_homogeneous_decomposition}. Assume we have shown the statement for $m\in\mathbb{N}$. The map
			\begin{align*}
			f\mapsto \mu\left(h+f\right)
			\end{align*}
			belongs to $\VConv(C;V,F)$ for all $h\in C$ by Lemma \ref{lemma:valuations-on-conv:continuity-addition}. Using Theorem \ref{maintheorem_homogeneous_decomposition}, we obtain 
			\begin{align*}
			\mu\left(h+tf\right)=\sum\limits_{i=0}^{n}t^i\mu_i\left(h,f\right),
			\end{align*}
			where $\mu_i:C^2\rightarrow F$ is an $i$-homogeneous, continuous, and dually epi-translation invariant valuation in the second argument and a dually epi-translation invariant valuation in the first. To see that $\mu_i$ is continuous in the first argument, apply the inverse of the Vandermonde matrix to the formula above to write $\mu_i(h,f)$ as a linear combination of elements of the form $\mu(h+tf)$ for a finite number of fixed values of $t>0$. \\ 
			The induction assumption implies that $(\lambda_1,\dots,\lambda_m)\mapsto \mu_i(\sum_{j=1}^{m}\lambda_j f_j,f)$ is a polynomial in $\lambda_j>0$, $1\le j\le m$. The claim follows.
		\end{proof}
		\begin{definition}
			A valuation $\mu\in \VConv(C;V,F)$ is called additive if $\mu(f+g)=\mu(f)+\mu(g)$ for all $f,g\in C$.
		\end{definition}
		By continuity, any additive valuation is $1$-homogeneous.
		\begin{theorem}
			\label{theorem_polarization_VConv}
			For $\mu\in \VConv_k(C;V,F)$, there exists a unique map $\bar{\mu}:C^k\rightarrow F$, called the \emph{polarization} of $\mu$, with the following properties:
			\begin{enumerate}
				\item $\bar{\mu}$ is additive and $1$-homogeneous in each coordinate,
				\item $\bar{\mu}$ is symmetric,
				\item $\mu(f)=\bar{\mu}(f,...,f)$ for all $f\in C$.
			\end{enumerate}
		\end{theorem}
		\begin{proof}
			We start by showing uniqueness: Using 1. and 3. we obtain
			\begin{align*}
				\mu\left(\sum\limits_{j=1}^{k}\lambda_jf_j\right)=\bar{\mu}\left(\sum\limits_{j=1}^{k}\lambda_jf_j,...,\sum\limits_{j=1}^{k}\lambda_jf_j\right)=\sum\limits_{j_1,...,j_k=1}^k\lambda_{j_1}...\lambda_{j_k}\bar{\mu}\left(f_{j_1},...,f_{j_k}\right).
			\end{align*}
			Differentiating and using 2., we obtain the formula
			\begin{align}
				\label{equation_definition_polarization}
				k!\bar{\mu}(f_1,...,f_k)=\frac{\partial}{\partial\lambda_1}\Big|_0...\frac{\partial}{\partial\lambda_k}\Big|_0\mu\left(\sum\limits_{j=1}^{k}\lambda_jf_j\right).
			\end{align}
			This shows uniqueness. To prove the existence of $\bar{\mu}$, we use Corollary \ref{corollary_polynomiality_valuations} to see that the right-hand side of \eqref{equation_definition_polarization} is actually well defined, so we can use this equation to define $\bar{\mu}$. \\
			Obviously the definition is symmetric in $f_1,...,f_n$. To see that $\bar{\mu}$ is additive in each coordinate, we thus only need to consider one coordinate: Setting 
			\begin{align*}
				F(t,s):=&\frac{1}{k!}\frac{\partial }{\partial \lambda_1}\Big|_0...\frac{\partial }{\partial \lambda_{k-1}}\Big|_0\mu\left(\sum\limits_{j=1}^{k-1}\lambda_j f_j+tf+sg\right),\\
				G(t):=& F(t,t),
			\end{align*}
			we obtain 
			\begin{align*}
				\bar{\mu}(f_1,...,f_{k-1},f+g)&=G'(0)=\frac{\partial F}{\partial t}\Big|_{(0,0)}+\frac{\partial F}{\partial s}\Big|_{(0,0)}
				=\bar{\mu}(f_1,...,f_{k-1},f)+\bar{\mu}(f_1,...,f_{k-1},g).
			\end{align*}
			Let us see that we can recover $\mu(f)$: 
			\begin{align*}
				\bar{\mu}(f,...,f)=\frac{1}{k!}\frac{\partial}{\partial\lambda_1}\Big|_0...\frac{\partial}{\partial\lambda_k}\Big|_0\mu\left(\sum\limits_{j=1}^{k}\lambda_jf\right)=\frac{1}{k!}\frac{\partial}{\partial\lambda_1}\Big|_0...\frac{\partial}{\partial\lambda_k}\Big|_0\left(\sum\limits_{j=1}^{k}\lambda_j\right)^k\cdot\mu(f)
			\end{align*} as $\mu$ is $k$-homogeneous. Thus $\bar{\mu}(f,...,f)=\mu(f)$.
		\end{proof}
		Note that the construction shows that $\bar{\mu}$ is a dually epi-translation invariant valuation in each coordinate. We will now show that $\bar{\mu}$ is jointly continuous. From the defining properties of $\bar{\mu}$ we deduce the following corollary.
		\begin{corollary}
			\label{corollary_boundedness_degree_polynomial_valuation}
			For $\mu\in \VConv_k(C;V,F)$, $m\in\mathbb{N}$ and $f_1,...,f_m\in C$, $\mu(\sum_{j=1}^{m}\lambda_jf_j)$ is a polynomial of degree at most $k$ in $\lambda_j> 0$.
		\end{corollary}
		\begin{corollary}
			\label{corollary_continuity_polarization}
			$\bar{\mu}:C^k\rightarrow\R$ is continuous for $\mu\in\VConv_k(C;V,F)$.
		\end{corollary}
		\begin{proof}
			Assume we are given sequences $(f_{i,j})_j$, $1\le i\le k$, in $C$ such that each sequence $(f_{i,j})_j$ converges to some $f_i\in C$. Then the polynomials $P_j(\lambda_1,...,\lambda_k):=\mu(\sum_{i=1}^{k}\lambda_if_{i,j})$ converge pointwise to $P(\lambda_1,...,\lambda_k):=\mu(\sum_{i=1}^{k}\lambda_if_{i})$ for $\lambda_i\ge0$. Note that the degree of $P_j$ is bounded by $k$ due to Corollary \ref{corollary_boundedness_degree_polynomial_valuation}, so we see that the coefficient in front of $\lambda_1\cdot...\cdot\lambda_k$ converges. Now the claim follows from the definition of $\bar{\mu}$ in the proof of Theorem \ref{theorem_polarization_VConv}.
		\end{proof}
		We close this section with an inequality that will be used in the construction of the Goodey-Weil embedding. It also shows that the map which associates the polarization to a given valuation is continuous with respect to the natural topologies.
		\begin{lemma}
			\label{lemma_continuity_valuation->polarization}
			There exists a constant $C_{k}>0$ depending on $0\le k\le n$ only such that the following holds: If $C\subset\Conv(V)$ is a regular cone containing $\Conv(V,\R)$ and if $K\subset C$ is compact, then 
			\begin{align*}
			\|\bar{\mu}\|_{F;K}:=\sup_{f_1,\dots,f_k\in K}|\bar{\mu}(f_1,\dots,f_k)|_F\le C_{k}\|\mu\|_{F;K'}
			\end{align*}
			for every semi-norm $|\cdot |_F$ on $F$ and all $\mu\in\VConv_k(V,F)$, where 
			\begin{align*}
			K':=\sum_{j=1}^k \bigcup\limits_{i=1}^{j+1}iK=\left\{\sum\limits_{j=1}^{k}f_j \ : \ f_j\in \bigcup\limits_{i=1}^{j+1}iK\right\}\subset C
			\end{align*} is compact. If $K$ is convex with $0\in K$, there exists a constant $C_k'>0$ independent of $K$ such that
			\begin{align*}
			\|\bar{\mu}\|_{F;K}\le C_{k}'\|\mu\|_{F;K}.
			\end{align*}
		\end{lemma}
		\begin{proof}
			For $f,g\in C$ and $\lambda\ge 0$
			\begin{align*}
			\mu(f+\lambda g)=\sum\limits_{i=0}^{k}\lambda^{i}i!(k-i)!\bar{\mu}(f[k-i],g[i]).
			\end{align*}
			We are only interested in the linear term. Plugging in $\lambda=1,\dots,k+1$ and setting $\mu'_i(f):=\mu(f+ig)$, we obtain a valuation $P_g(\mu)=(\mu_1',\dots,\mu_{k+1}')\in\VConv(C;V,F^{k+1})$. Let $\pi_i:F^{k+1}\rightarrow F$ denote the $i$-th projection and $S_k$ the Vandermonde matrix with entries corresponding to $(1,\dots,k+1)$. Then 
			\begin{align*}
			\bar{\mu}(f[k-1],g)=\frac{1}{(k-1)!}\pi_1[S_k^{-1}P_g(\mu)(f)].
			\end{align*}
			If we equip $F^{k+1}$ with the family of semi-norms $|(v_1,\dots,v_{k+1})|_F:=\max_{i=1,\dots,k+1}|v_i|_F$, and denote by $\|S_k^{-1}\|_\infty$ the operator norm of $S_k^{-1}:\R^{k+1}\rightarrow \R^{k+1}$ with respect to the maximum norm on $\R^{k+1}$, we obtain
			\begin{align*}
			|\bar{\mu}(f[k-1],g)|_F\le\frac{1}{(k-1)!}\|S_k^{-1}\|_\infty \ |P_g(\mu)(f)|_F.
			\end{align*}
			For $g\in K$, $P_g:\VConv(C;V,F)\rightarrow \VConv(C;V,F^{k+1})$ satisfies
			\begin{align*}
			|P_g(\mu)(f)|_F=\max\limits_{i=1,\dots,k+1}|\mu(f+ig)|_F\le \sup\left\{|\mu(f+\tilde{g})|_F: \tilde{g}\in\bigcup\limits_{i=1}^{k+1}iK\right\}
			\end{align*}
			independent of $g\in K$. Iterating this construction starting with the $k-1$-homogeneous valuation $\nu(f):=\bar{\mu}(f[k-1],g)$, we see that there exists $C_{k}>0$ depending on $k$ only such that for $f_1,\dots,f_k\in K$
			\begin{align*}
			|\bar{\mu}(f_1,\dots,f_k)|_F\le C_{k}\sup\left\{|\mu(\tilde{g})|_F: \tilde{g}\in \sum_{j=1}^{k} \bigcup\limits_{i=1}^{j+1}iK\right\}= C_{k}\|\mu\|_{F,K'}
			\end{align*}
			for every semi-norm $|\cdot|_F$ on $F$. Also note that $K'$ is compact, as it is the image of a compact subset under the addition map, which is continuous on $C$ by Lemma \ref{lemma:valuations-on-conv:continuity-addition}.\\
			If $K$ is convex with $0\in K$, then $K'\subset (k+1)^2K$, so 
			\begin{align*}
			\|\mu\|_{F;K'}=\sup\limits_{g\in K'}|\mu(g)|_F\le \sup\limits_{g\in (k+1)^2K}|\mu(g)|_F\le ((k+1)^2)^k\sup\limits_{g\in K}|\mu(g)|=(k+1)^{2k}\|\mu\|_{F;K},
			\end{align*}
			i.e. we can choose $C_k':=(k+1)^{2k}C_k$.
		\end{proof}

\section{Goodey-Weil embedding}
	\label{section_Goodey-Weil-embedding}
	\subsection{Construction and basic properties}
		In this section, we will assume that $V$ carries an euclidean structure. Let $C^2_b(V)$ denote the Banach space of twice differentiable functions with finite $C^2$-norm
		\begin{align*}
			\|\phi\|_{C^2_b(V)}:=&\|\phi\|_\infty+\|\nabla \phi\|_\infty+\|H_\phi\|_\infty
			=\sup\limits_{x\in V}|\phi(x)|+\sup\limits_{x\in V}|\nabla f(x)|+\sup\limits_{x\in V,v\in S(V)}|\langle H_\phi(x)v,v\rangle|,
		\end{align*}
		where $H_\phi$ denotes the Hessian matrix of a twice differentiable function $\phi:V\rightarrow \R$ and $S(V)\subset V$ is the unit sphere. Let us also set $c(A):=\sup_{x\in A}\frac{|x|^2}{2}+1$ for any compact subset $A\subset V$.
		\begin{lemma}
			\label{lemma_difference_C2_convex_functions}
			For every $\phi\in C_b^2(V)$ there exist two convex functions $f,h\in\Conv(V,\R)$ such that $f-h=\phi$ and  such that $\|f|_A\|_{\infty},\|h|_A\|_{\infty}\le c(A)\|\phi\|_{C_b^2(V)}$ for all compact subsets $A\subset V$. These functions can be chosen in $C^\infty(V)$.
		\end{lemma}
		\begin{proof}
			Take $f(x):=c\frac{|x|^2}{2}+\phi(x)$, $h(x)=c\frac{|x|^2}{2}$, where $c:= \|\phi\|_{C_b^2(V)}$. Then $f$ and $h$ are convex, as their Hessians are positive semi-definite. In addition
			\begin{align*}
			\|h|_A\|_{\infty},\|f|_A\|_{\infty}\le c\cdot\sup_{x\in A}\frac{|x|^2}{2}+\|\phi\|_\infty\le \left(\sup_{x\in A}\frac{|x|^2}{2}+1\right)\|\phi\|_{C_b^2(V)}=c(A)\|\phi\|_{C_b^2(V)}.
			\end{align*}
		\end{proof}
		To every $\mu\in \VConv_k(C;V,F)$ we can associate a $k$-multilinear functional $\tilde{\mu}$ on formal differences of convex functions: Assume that $h_1+\phi_1=f_1$, $\dots$, $h_1+\phi_k=f_k$ for convex functions $f_1,\dots,f_k,h_1,\dots,h_k\in C$. Using the polarization $\bar{\mu}$ from Theorem \ref{theorem_polarization_VConv}, we define the following functionals inductively for arbitrary convex functions $g_1,\dots,g_k\in\Conv(V,\R)$ and $1\le i\le k-1$:
		\begin{align*}
		\mu^{(1)}(\phi_1,g_2,\dots,g_k):=&\bar{\mu}(f_1,g_2,\dots,g_k)-\bar{\mu}(h_1,g_2,\dots,g_k),\\
		\mu^{(i+1)}(\phi_1,\dots,\phi_{i+1},g_{i+2},\dots,g_k):=&\mu^{(i)}(\phi_1,\dots,\phi_i,f_{i+1},g_{i+1},\dots,g_k)\\
		&-\mu^{(i)}(\phi_1,\dots,\phi_i,h_{i+1},g_{i+1},\dots,g_k).
		\end{align*}
		Then we set $\tilde{\mu}(\phi_1,\dots,\phi_k):=\mu^{(k)}(\phi_1,\dots,\phi_k)$. It is easy to see that this is equivalent to
		\begin{align}
		\label{equation_representation_GW_using_polarization}
		\tilde{\mu}(\phi_1,..,\phi_k)=\sum\limits_{l=0}^k(-1)^{k-l}\frac{1}{(k-l)!l!}\sum_{\sigma\in S_k}\bar{\mu}(f_{\sigma(1)},\dots,f_{\sigma(l)},h_{\sigma(l+1)},\dots,h_{\sigma(k)}),
		\end{align}
		where $\bar{\mu}$ denotes the polarization of $\mu$ from Theorem \ref{theorem_polarization_VConv}. Thus $\tilde{\mu}$ is also symmetric.\\
		Using the additivity of $\bar{\mu}$ in each argument, one readily verifies that this definition only depends on the functions $\phi_1,\dots,\phi_k$ (and not the special choices of $f_i$ and $h_i$) and that this functional is multilinear. 
		
		By Lemma \ref{lemma_smooth function difference of convex functions}, $C_c^2(V)$ is contained in the space generated by differences of elements of $\Conv(V,\R)$. For the construction of the Goodey-Weil embedding, we will thus consider the restricted map
		\begin{align*}
		\tilde{\mu}:C_b^2(V)^k\rightarrow F.
		\end{align*}
		Given functions $\phi_1,\dots,\phi_k\in C^2_b(V)$, we take the special convex functions $f_1$,\dots, $f_k$, $h_1$,\dots,$h_k$ in $\Conv(V,\R)$ with $\phi_i=f_i-h_i$ satisfying the inequality in Lemma \ref{lemma_difference_C2_convex_functions}. The set $K$ of all convex functions that are bounded by $c(A)$ on every compact set $A$ (as defined above) is compact in $\Conv(V,\R)$ by Proposition \ref{proposition_compactness_Conv}, so it is also compact in $C$. Note that the functions
		\begin{align*}
		&\tilde{f}_i:=\frac{f_i}{\|\phi_i\|_{C^2_b(V)}}, &&\tilde{h}_i:=\frac{h_i}{\|\phi_i\|_{C^2_b(V)}}
		\end{align*}
		belong to $K$ by construction. As $K$ is also convex with $0\in K$, Lemma \ref{lemma_continuity_valuation->polarization} and Equation \eqref{equation_representation_GW_using_polarization} imply
		\begin{align}
		\notag
		|\tilde{\mu}(\phi_1,\dots,\phi_k)|_F&=|\sum\limits_{l=0}^k(-1)^{k-l}\frac{1}{(k-l)!l!}\sum_{\sigma\in S_k}\bar{\mu}(f_{\sigma(1)},\dots,f_{\sigma(l)},h_{\sigma(l+1)},\dots,h_{\sigma(k)})|_F\\
		\notag
		&\le \sum\limits_{l=0}^k\frac{1}{(k-l)!l!}\sum_{\sigma\in S_k}|\bar{\mu}(f_{\sigma(1)},\dots,f_{\sigma(l)},h_{\sigma(l+1)},\dots,h_{\sigma(k)})|_F\\
		\notag
		&= \sum\limits_{l=0}^k\frac{1}{(k-l)!l!}\sum_{\sigma\in S_k}|\bar{\mu}(\tilde{f}_{\sigma(1)},\dots,\tilde{f}_{\sigma(l)},\tilde{h}_{\sigma(l+1)},\dots,\tilde{h}_{\sigma(k)})|_F\cdot \prod\limits_{i=1}^ k\|\phi_i\|_{C_b^2(V)}\\
		\label{equation_continuity_GW}
		&\le c_{k}\|\mu\|_{F;K}\cdot \prod\limits_{i=1}^ k\|\phi_i\|_{C_b^2(V)}
		\end{align}
		for any continuous semi-norm $|\cdot|_F$ on $F$ for some constant $c_{k}>0$ depending on $k$ only (we can choose $2^k$ times the constant $C'_k$ from Lemma \ref{lemma_continuity_valuation->polarization}). Here we have used that $\bar{\mu}$ is $1$-homogeneous in each argument. As $\tilde{\mu}$ is multilinear, this inequality implies that $\tilde{\mu}$ is continuous. In particular, we can apply the L. Schwartz kernel theorem (see \cite{Gask:Schwartz_kernel} for a simple proof).
		\begin{theorem}[L. Schwartz kernel theorem]
			Let $F$ be a complete locally convex vector space and let $V,W$ be finite dimensional real vector spaces. For every continuous bilinear map
			\begin{align*}
				b:C^\infty_c(V)\times C^\infty_c(W)\rightarrow F
			\end{align*}
			there exists a unique continuous linear map
			\begin{align*}
				B:C^\infty_c(V\times W)\rightarrow F
			\end{align*}
			such that $B(f\otimes h)=b(f,h)$ for all $f\in C^\infty_c(V)$, $h\in C^\infty_c(W)$.
		\end{theorem}
		Let us denote by $\mathcal{D}'(V,F)$ the space of all distributions on $V$ with values in a locally convex vector space $F$, i.e. the space of all continuous linear functionals $\phi:C^\infty_c(V)\rightarrow F$. Applying the L. Schwartz kernel theorem to the functional $\tilde{\mu}$ for $\mu\in \VConv_k(C;V,F)$, we obtain the following notion:
		\begin{definition}
			\label{definition_Goddey_Weil_embedding_VCONV}
			Let $F$ be a locally convex vector space, $1\le k\le n$. To every $\mu\in\VConv_k(C;V,F)$ we associate the distribution $\GWVConv(\mu)\in \mathcal{D}'(V^k,\bar{F})$ determined by
			\begin{align*}
				\GWVConv(\mu)(\phi_1\otimes...\otimes\phi_k)=\tilde{\mu}(\phi_1,...,\phi_k)
			\end{align*}
			for $\phi_1,...,\phi_k\in C_c^\infty(V)$. The distribution will be called the \emph{Goodey-Weil distribution} of $\mu$.
		\end{definition}
		\begin{lemma}
			\label{lemma:valuations-on-conv:Goodey-Weil-as-derivative}
			Let $\mu\in\VConv_k(C;V,F)$ and let $f\in\Conv(V,\R)$ be a strictly convex function. If $\phi_1,\dots,\phi_k\in C^\infty_c(V)$, then $\mu(f+\sum_{i=1}^{k}\delta_i\phi_i)$ is a polynomial in $\delta_i\ge0$ for all $\delta_i$ small enough and
			\begin{align*}
			\GWVConv(\mu)[\phi_1\otimes\dots\otimes\phi_k]=\frac{1}{k!}\frac{\partial}{\partial \delta_1}\Big|_0\dots \frac{\partial}{\partial \delta_k}\Big|_0\mu\left(f+\sum\limits_{i=1}^{k}\delta_i\phi_i\right).
			\end{align*}
		\end{lemma}
		\begin{proof}
			Note that the strict convexity implies that $f+\sum_{i=1}^{k}\delta_i\phi_i$  is a convex function for all $\delta_i$ sufficiently small. Let us consider the multilinear functional $\tilde{\mu}$ from Section \ref{section_Goodey-Weil-embedding}. The construction of $\tilde{\mu}$ implies $\tilde{\mu}(h,\dots,h)=\mu(h)$ for $h\in\Conv(V,\R)$, so
			\begin{equation*}
			\mu\left(f+\sum_{i=1}^{k}\delta_i\phi_i\right)=\tilde{\mu}\left(f+\sum_{i=1}^{k}\delta_i\phi_i,\dots,f+\sum_{i=1}^{k}\delta_i\phi_i\right).
			\end{equation*}
			In particular, the left hand side is a polynomial in $\delta_i$ for all $\delta_i$ small enough and we calculate
			\begin{align*}
			\frac{1}{k!}\frac{\partial}{\partial \delta_1}\Big|_0\dots \frac{\partial}{\partial \delta_k}\Big|_0\mu\left(f+\sum_{i=1}^{k}\delta_i\phi_i\right)=&\frac{1}{k!}\frac{\partial}{\partial \delta_1}\Big|_0\dots \frac{\partial}{\partial \delta_k}\Big|_0\tilde{\mu}\left(f+\sum_{i=1}^{k}\delta_i\phi_i,\dots,f+\sum_{i=1}^{k}\delta_i\phi_i\right)\\
			=&\tilde{\mu}(\phi_1,\dots,\phi_k)\\
			=&\GWVConv(\mu)[\phi_1\otimes\dots\otimes\phi_k],
			\end{align*}
			where we have used that $\tilde{\mu}$ is multilinear and symmetric. 
		\end{proof}

	\subsection{Diagonality of the support of the Goodey-Weil distributions}
	\label{section_support_GW_diagonal}
	The following theorem is an adaption of the proof of the corresponding statement for the Goodey-Weil embedding for translation invariant valuations on convex bodies (see \cite{Alesker:McMullenconjecture}). It also proves the second part of Theorem \ref{maintheorem_GW}.
	\begin{theorem}
		\label{theorem_support_GW_diagonal}
		Let $F$ be a locally convex vector space. For $\mu\in\VConv_k(C;V,F)$ the support of $\GWVConv(\mu)$ is contained in the diagonal in $V^k$.
	\end{theorem}
	\begin{proof}
		Let us assume that $V$ carries a Euclidean structure. Using a partition of unity, it is sufficient to show that $\GWVConv(\mu)(h_1\otimes\dots \otimes h_k)=0$ if $h_1,\dots,h_k\in C^\infty_c(V)$ are smooth functions satisfying $\supp h_i\subset U_\epsilon (a_i)$, where $a_i\in V$, $1\le i\le k$, are points with $U_{\epsilon}(a_i)\cap U_{\epsilon}(a_j)=\emptyset$ for $i\ne j$ and some fixed $\epsilon>0$.\\
		First, there exists $\delta>0$ such that the function $x\mapsto1+|x|^2+\sum_{i=1}^{k}\delta_i h_i$ is convex and non-negative for all $|\delta_i|\le \delta$. Set 
		\begin{align*}
		H(x):=1+|x|^2+ \sum\limits_{i=3}^{k}\delta_i h_i
		\end{align*} and choose an affine hyperplane that separates $U_\epsilon(a_1)$ and $U_\epsilon(a_2)$. This plane is given by the equation $\langle y,x-x_0\rangle=0$ for some $y,x_0\in V$. We can choose $y$ such that $U_\epsilon(a_1)$ is contained in the positive half space with respect to the normal $y$. Define the convex functions
		\begin{align*}
		G_\pm(x):=&\max(0,\pm\langle y,x-x_0\rangle)
		=\begin{cases}
		0 & \pm\langle y,x-x_0\rangle\le 0,\\
		\pm \langle y,x-x_0\rangle & \pm\langle y,x-x_0\rangle>0.
		\end{cases}
		\end{align*} 
		As $G_\pm$ is positive on the supports of $h_1$ and $h_2$ respectively, we can rescale $y$ such that $G_+$ is larger than $H+\delta_1h_1$ on the support of $h_1$ and $G_-$ is larger than $H+\delta_2h_2$ on the support of $h_2$ for all $|\delta_i|\le\delta$.
		Now set $\tilde{H}_+:=\max(H+\delta_1h_1,G_+)$ and $\tilde{H}_-:=\max(H+\delta_1h_1,G_-)$. Then $\tilde{H}_+$ and $\tilde{H}_-$ are convex functions with
		\begin{align*}
		\tilde{H}_+(x)=&\begin{cases}
		H(x)+\delta_1h_1(x) &  \langle y,x-x_0\rangle\le 0,\\
		\max(H(x)+\delta_1h_1(x),\langle y,x-x_0\rangle )& \langle y,x-x_0\rangle > 0,
		\end{cases}\\
		\tilde{H}_-(x)=&\begin{cases}
		H(x)+\delta_1h_1(x) & \langle y,x-x_0\rangle\ge 0,\\
		\max(H(x)+\delta_1h_1(x),-\langle y,x-x_0\rangle )& \langle y,x-x_0\rangle < 0.
		\end{cases}
		\end{align*}
		In particular $\min (\tilde{H}_+,\tilde{H}_-)=\tilde{H}:=H+\delta_1 h_1$ is convex. 
		As the support of $h_1$ is contained in the positive half space with respect to $y$, we see that in fact
		\begin{align*}
		\tilde{H}_+(x)=&\begin{cases}
		H(x) &  \langle y,x-x_0\rangle\le 0,\\
		\max(H(x),\langle y,x-x_0\rangle )& \langle y,x-x_0\rangle > 0,
		\end{cases}\\
		\tilde{H}_-(x)=&\begin{cases}
		H(x)+\delta_1h_1(x) & \langle y,x-x_0\rangle\ge 0,\\
		\max(H(x),-\langle y,x-x_0\rangle )& \langle y,x-x_0\rangle < 0.
		\end{cases}
		\end{align*}
		Thus $\max(\tilde{H}_+,\tilde{H}_-)=\max (H,|\langle y,\cdot-x_0\rangle|)$.\\
		
		Let us also define $H_\pm:=\max(H,G_\pm)$. Then $H_+=\tilde{H}_+$, and, using the non-negativity of $H$, it is easy to see that $\min(H_+,H_-)=H$ and $\max(H_+,H_-)=\max (H,|\langle y,\cdot-x_0\rangle|)=\max(\tilde{H}_+,\tilde{H}_-)$.\\
		The valuations property implies
		\begin{align*}
		\mu(\tilde{H}_+)+\mu(\tilde{H}_-)=\mu(\max(\tilde{H}_+,\tilde{H}_-))+\mu(\min (\tilde{H}_+,\tilde{H}_-)),\\
		\mu(H_+)+\mu(H_-)=\mu(\max(H_+,H_-))+\mu(\min (H_+,H_-)).
		\end{align*}
		Thus using $\max (H_+,H_-)=\max (\tilde{H}_+,\tilde{H}_-)$,  and $\tilde{H}_+=H_+$, we obtain
		\begin{align*}
		\mu(\tilde{H}_-)-\mu(H_-)=\mu(\min(\tilde{H}_+,\tilde{H}_-))-\mu(\min(H_+,H_-))
		\end{align*}
		by subtracting the two equations. Plugging in the relations for the minima, we arrive at
		\begin{align*}
		\mu(\tilde{H}_-)-\mu(H_-)=\mu(\tilde{H})-\mu(H).
		\end{align*} 
		Set 
		\begin{align*}
		\Delta(x):=& H_-(x)=\begin{cases}
		H(x) &  \langle y,x-x_0\rangle\ge0,\\
		\max(H(x),-\langle y,x-x_0\rangle )& \langle y,x-x_0\rangle < 0,
		\end{cases}\\
		\tilde{\Delta}(x):=& \tilde{H}_-(x)=\begin{cases}
		H(x) +\delta_1h_1(x)&  \langle y,x-x_0\rangle\ge0,\\
		\max(H(x),-\langle y,x-x_0\rangle )& \langle y,x-x_0\rangle < 0,
		\end{cases}
		\end{align*}
		to rewrite the previous equation as 
		\begin{align*}
		\mu(\tilde{\Delta})-\mu(\Delta)=\mu(\tilde{H})-\mu(H)=\mu(H+\delta_1h_1)-\mu(H).
		\end{align*}
		Now, if we replace $H$ by $H+\delta_2h_2$ and repeat the argument, the convex functions $\Delta'$ and $\tilde{\Delta}'$ defined by
		\begin{align*}
		\Delta'(x)=& \begin{cases}
		H(x) +\delta_2h_2(x)&  \langle y,x-x_0\rangle\ge0,\\
		\max(H(x)+\delta_2h(x),-\langle y,x-x_0\rangle )& \langle y,x-x_0\rangle < 0,
		\end{cases}\\
		\tilde{\Delta}'(x)=& \begin{cases}
		H(x) +\delta_1h_1(x)+\delta_2h_2(x)&  \langle y,x-x_0\rangle\ge0,\\
		\max(H(x)+\delta_2h_2(x),-\langle y,x-x_0\rangle )& \langle y,x-x_0\rangle < 0,
		\end{cases}
		\end{align*}
		satisfy
		\begin{align*}
		\mu(\tilde{\Delta}')-\mu(\Delta')=\mu(H+\delta_1h_1+\delta_2h_2)-\mu(H+\delta_2h_2).
		\end{align*}
		However, the support of $h_2$ is contained in the negative half space and $H+\delta_2 h_2$ is smaller than $-\langle y,\cdot-x_0\rangle$ on the support of $h_2$. Thus $\Delta'=\Delta$ and $\tilde{\Delta}'=\tilde{\Delta}$, and we obtain the equation
		\begin{align*}
		\mu(H+\delta_1h_1)-\mu(H)=\mu(\tilde{\Delta})-\mu(\Delta)=\mu(H+\delta_1h_1+\delta_2h_2)-\mu(H+\delta_2h_2)
		\end{align*}
		for all $\delta_i$ with $|\delta_i|<\delta$. Both the left and the right hand side are polynomials in $\delta_i$ for all $\delta_i$ small enough by Lemma \ref{lemma:valuations-on-conv:Goodey-Weil-as-derivative}, but the left hand side is independent of $\delta_2$. Thus Lemma \ref{lemma:valuations-on-conv:Goodey-Weil-as-derivative} implies
		\begin{align*}
		0=&\frac{1}{k!}\frac{\partial}{\partial \delta_1}\Big|_0\dots\frac{\partial}{\partial \delta_k}\Big|_0\left[\mu(H+\delta_1h_1+\delta_2h_2)-\mu(H+\delta_2h_2)\right]\\
		=&\frac{1}{k!}\frac{\partial}{\partial \delta_1}\Big|_0\dots\frac{\partial}{\partial \delta_k}\Big|_0\mu(H+\delta_1h_1+\delta_2h_2)\\
		=&\GWVConv(\mu)[h_1\otimes\dots\otimes h_k].
		\end{align*} 
	\end{proof}
	\subsection{Compactness of the support}
		\begin{lemma}
			\label{lemma_smooth function difference of convex functions}
			For every $\phi\in C^2(V)$ there exists a convex function $h\in\Conv(V,\R)$ with the following property: If $\psi\in C^\infty(V)$ is a function with $\|\psi\|_{C^2(B_j)}\le \|\phi\|_{C^2(B_j)}$ for all $j\in\mathbb{N}$, then $h+\psi$ is convex.
		\end{lemma}
		\begin{proof}
			Assume that we are given $\phi$ and let $\psi$ be an arbitrary function with the property stated above. Let us inductively define a sequence of convex functions $h_j\in\Conv(V,\R)$. Set $c_j:=\|\phi\|_{C^2(B_{j+1})}$. Then $c_j\ge\sup_{x\in B_{j+1},v\in S(V)}\langle H_\psi(x)v,v\rangle $ for all $j\in\mathbb{N}$.\\
			For $j=1$ define $h_1$ by $h_1(x):=c_1\frac{|x|^2}{2}$. As its Hessian is positive semi-definite, $h_1+\psi$ is convex on $B_2$.\\
			Assume that we have already constructed $h_j$. Then the Hessian of $c_{j+1}\frac{|x|^2}{2}+\psi$ is positive semi-definite on $B_{j+2}$, so this function is convex on $B_{j+2}$. We set 
			\begin{align*}
			h_{j+1}(x):=&\max\left(c_{j+1}\frac{|x|^2-j^2}{2},0\right)+h_j(x).
			\end{align*}
			Thus $h_{j+1}$ is a finite-valued convex function for all $j\in\mathbb{N}$, that coincides with $h_j$ on $B_j$. 
			We deduce that for each point $x\in V$ the sequence $(h_j(x))_j$ becomes constant for large $j$ and thus $(h_j)_j$ converges pointwise to a function $h\in\Conv(V,\R)$. By Proposition \ref{proposition_convergence_finite_convex_functions}, this implies that this sequence epi-converges to $h$.\\
			It remains to see that $h+\psi$ is convex. Observe that for every point $x\in V$ there exists an open, convex neighborhood such that the restriction of $h+\psi$ to this neighborhood is convex, i.e. $h+\psi$ is a locally convex function: On $U_{j+1}\setminus B_{j-1}$
			\begin{align*}
			h(x)+\psi(x)=&\max\left(c_{j+1}\frac{|x|^2-j^2}{2},0\right)+h_j(x)+\psi(x)\\
			=&\max\left(c_{j+1}\frac{|x|^2-j^2}{2},0\right)+h_{j-1}(x)+c_{j}\frac{|x|^2-(j-1)^2}{2}+\psi(x),
			\end{align*}
			where $c_{j}\frac{|x|^2-(j-1)^2}{2}+\psi(x)$ is locally convex on this set, as its Hessian is positive semi-definite. Obviously, the other two functions are locally convex on this set as well, so the same applies to $h+\psi$.\\
			As any locally convex function defined on $V$ is convex, $h+\psi\in\Conv(V,\R)$.
		\end{proof}
		We will now prove the remaining parts of Theorem \ref{maintheorem_GW}.
		\begin{theorem}
			\label{theorem_compact_support_GW}
			Let $F$ be a locally convex vector space admitting a continuous norm. Then the following holds: For every $\mu\in\VConv_k(C;V,F)$ the distribution $\GWVConv(\mu)\in\mathcal{D}'(V^k,\bar{F})$ has compact support and is uniquely defined by the following property: If $f_1,...,f_k\in \Conv(V,\R)\cap C^\infty(V)$ then
			\begin{align}
			\label{equation_defining_equality_GW}
			\GWVConv(\mu)(f_1\otimes...\otimes f_k)=\bar{\mu}(f_1,...,f_k),
			\end{align}
			where $\bar{\mu}$ is the polarization of $\mu$. 
		\end{theorem}
		\begin{proof}
			Uniqueness follows directly from Equation \eqref{equation_defining_equality_GW}, as every function $\phi\in C^\infty_c(V)$ can be written as a difference of two smooth convex functions due to Lemma \ref{lemma_difference_C2_convex_functions} and a distribution on $V^k$ is uniquely determined by its values on functions of the form $\phi_1\otimes\dots\otimes\phi_k$ for $\phi_1,\dots,\phi_k\in C^\infty_c(V)$ by the L. Schwartz kernel theorem.\\
			Let us assume that $\GWVConv(\mu)$ does not have compact support and let $\|\cdot\|$ denote a continuous norm on $F$. Then we can inductively define a sequence of functions $(\phi_i^j)_j$ in $C_c^\infty(V)$ for each $1\le i\le k$ and a strictly increasing sequence $(r_j)_j$ of positive real numbers with $\lim\limits_{j\rightarrow\infty} r_j=\infty$ with the following properties:
			\begin{enumerate}
				\item For each $1\le i\le k$ the functions $(\phi_i^j)_j$ have pairwise disjoint support.
				\item The support of $\phi_i^j$ is contained in $V\setminus  B_{r_j}(0)$ for all $j\in\mathbb{N}$, $1\le i\le k$.
				\item $\|\GWVConv(\mu)(\phi_1^j\otimes\dots\otimes \phi_k^j)\|=\|\tilde{\mu}(\phi_1^j,\dots,\phi_k^j)\|=1$.
			\end{enumerate}  
			To see this, assume that we have constructed $\phi_1^j,\dots,\phi_k^j$ as well as $r_j>0$. First choose $r_{j+1}>r_j+1$ such that the restriction of $\GWVConv(\mu)$ to $[U_{r_{j+1}}\setminus B_{r_j}]^k\subset V^k$ does not vanish. This is possible, as the support of $\GWVConv(\mu)$ is contained in the diagonal due to Theorem \ref{theorem_support_GW_diagonal}. Then take $\phi_1^{j+1},\dots,\phi_k^{j+1}\in C^\infty(U_{r_{j+1}}\setminus B_{r_j})$ with $\GWVConv(\mu)[\phi_1^{j+1}\otimes\dots\otimes\phi_k^{j+1}]\ne0$ and rescale one function by an appropriate constant.\\
			Note that $\phi_i:=\sum_{j=1}^{\infty}\phi_{i}^j\in C^\infty(V)$ is well defined as this sum is locally finite. More precisely, the supports of the functions $\phi_i^j$ are disjoint for each $1\le i\le k$, so $\|\phi_i^j\|_{C^2(B_N)}\le \|\phi_i\|_{C^2(B_N)}$ for all $N\in\mathbb{N}$ and $j\in\mathbb{N}$. Applying Lemma \ref{lemma_smooth function difference of convex functions} to the functions $\phi_i$, we find convex functions $h_1,\dots, h_k\in\Conv(V,\R)$ such that for $1\le i\le k$ the function $f_i^j:=h_i+\phi_i^j$ is convex for all $j\in\mathbb{N}$. By Equation \eqref{equation_representation_GW_using_polarization}
			\begin{align*}
			\tilde{\mu}(\phi_1^j,\dots,\phi_k^j)=&\sum\limits_{l=0}^k(-1)^{k-l}\frac{1}{(k-l)!l!}\sum_{\sigma\in S_k}\bar{\mu}\left(f^j_{\sigma(1)},\dots,f^j_{\sigma(l)},h_{\sigma(l+1)},\dots,h_{\sigma(k)}\right).
			\end{align*}
			As $f_i^j\rightarrow h_i$ uniformly on compact subsets for all $1\le i\le k$, the joint continuity of the polarization $\bar{\mu}$ from Corollary \ref{corollary_continuity_polarization} and the continuity of the norm $\|\cdot\|$ imply
			\begin{align*}
			\lim\limits_{j\rightarrow\infty}\|\tilde{\mu}(\phi_1^j,\dots,\phi_k^j)\|
			=&\|\sum\limits_{l=0}^k(-1)^{k-l}\frac{1}{(k-l)!l!}\sum_{\sigma\in S_k}\bar{\mu}\left(h_{\sigma(1)},\dots,h_{\sigma(l)},h_{\sigma(l+1)},\dots,h_{\sigma(k)}\right)\|\\
			=&\|(-1)^k\sum\limits_{l=0}^{k}(-1)^l\frac{k!}{(k-l)!l!}\bar{\mu}(h_1,\dots,h_k)\|\\
			=&\|(-1)^k\bar{\mu}(h_1,\dots,h_k)\sum\limits_{l=0}^{k}(-1)^l\binom{k}{l}\|\\
			=&\|\bar{\mu}(h_1,\dots,h_k)\|\cdot0=0.
			\end{align*}
			We arrive at a contradiction to $\|\tilde{\mu}(\phi_1^j,\dots,\phi_k^j)\|=1$ for all $j\in\mathbb{N}$, so the distribution $\GWVConv(\mu)$ must have compact support.\\
			It remains to see that $\GWVConv(\mu)(f_1\otimes\dots\otimes f_k)=\bar{\mu}(f_1,\dots,f_k)$ for all convex functions $f_i\in \Conv(V,\R)\cap C^\infty(V)$. Take a sequence of functions $\phi_j\in C_c^\infty(V)$ with $\phi_j\equiv 1$ on $B_j(0)$ and such that $\|\phi_j\|_{C^2(V)}\le C$ for all $j\in\mathbb{N}$ and some $C>0$. Such a sequence can be constructed by setting $\phi_j(x):=\psi(\frac{x}{j})$ for $\psi\in C^\infty_c(V)$ with $\psi\equiv 1$ on $B_1(0)$. As the support of $\GWVConv(\mu)$ is compact, we obtain $N\in\mathbb{N}$ such that
			\begin{align*}
			\GWVConv(\mu)(f_1\otimes\dots\otimes f_k)=\GWVConv(\mu)(\phi_j f_1\otimes\dots\otimes \phi_j f_k)\quad\forall j\ge N.
			\end{align*}
			Using the Leibniz-rule, we see that there exists $C'>0$ such that for any compact set $A\subset V$  the inequality $\|\phi_jf_i\|_{C^2(A)}\le C'\|f_i\|_{C^2(A)}$ holds for all $j\in\mathbb{N}$. Now take the function $h_i$ from Lemma \ref{lemma_smooth function difference of convex functions} for the function $\phi=C'f_i$. Then $h_i+\phi_j f_i$ is convex for all $j\in\mathbb{N}$ and $h_i+\phi_jf_i$ converges to $h_i+f_i$ uniformly on compact subsets, i.e. in $\Conv(V,\R)$. Plugging in the definition of $\tilde{\mu}$ and using the joint continuity of the polarization $\bar{\mu}$, we obtain
			\begin{align*}
			& \GWVConv(\mu)(f_1\otimes\dots\otimes f_k)=\lim\limits_{j\rightarrow\infty}\GWVConv(\mu)(\phi_j f_1\otimes\dots\otimes\phi_j f_k)\\
			=&\lim\limits_{j\rightarrow\infty}\sum\limits_{l=0}^k(-1)^{k-l}\frac{1}{(k-l)!l!}\sum_{\sigma\in S_k}\bar{\mu}(h_{\sigma(1)}+\phi_jf_{\sigma(1)},\dots,h_{\sigma(l)}+\phi_jf_{\sigma(l)},h_{\sigma(l+1)},\dots,h_{\sigma(k)})\\
			=&\sum\limits_{l=0}^k(-1)^{k-l}\frac{1}{(k-l)!l!}\sum_{\sigma\in S_k}\bar{\mu}(h_{\sigma(1)}+f_{\sigma(1)},\dots,h_{\sigma(l)}+f_{\sigma(l)},h_{\sigma(l+1)},\dots,h_{\sigma(k)})\\
			=&\bar{\mu}(f_1,\dots,f_k),
			\end{align*}
			where we have used the additivity of $\bar{\mu}$ in the last step.
		\end{proof}
		\begin{corollary}
			\label{corollary_GW_injective}
			$\GWVConv:\VConv_k(C;V,F)\rightarrow \mathcal{D}'(V^k,\bar{F})$ is injective.
		\end{corollary}
		\begin{proof}
			Assume first that $F$ admits a continuous norm. As $\Conv(V,\R)\cap C^\infty(V)$ is dense in $C$ due to Proposition \ref{proposition_properties_LIpschitz_regularization}, the claim follows from Theorem \ref{theorem_compact_support_GW} and the continuity of $\mu$.\\
			If $F$ is an arbitrary locally convex vector space, then the definition of $\GWVConv$ implies
			\begin{align*}
			\lambda\circ \GWVConv(\mu)=\GWVConv(\lambda\circ\mu) \quad \forall \lambda\in \bar{F}'\cong F',
			\end{align*}
			where $F'$ denotes the topological dual of $F$. 
			In particular, $\GWVConv(\mu)=0$ if and only if $\GWVConv(\lambda\circ\mu)=0$ for all $\lambda\in F'$. By the previous discussion, $\GWVConv(\lambda\circ\mu)=0$ implies $\lambda\circ\mu=0$. If this holds for all $\lambda\in F'$, we obtain $\mu=0$, as $F$ is locally convex.
		\end{proof}
		Let us contrast the compactness of the support of $\GWVConv(\mu)$ for valuations with values in a locally convex vector space admitting a continuous norm with the more general case. Consider the following example: As a special case of the Hessian measures examined by Colesanti, Ludwig and Mussnig in \cite{Colesanti_Ludwig_Mussnig:Hessian_valuations}, we can consider the valuation $\Phi_n\in\VConv_n(V,\mathcal{M}(V))$ with values in the space $\mathcal{M}(V)$ of signed Radon measures on $V$ that extends 
		\begin{align*}
		\tilde{\Phi}_n(f)[B]:=\int_V 1_B(x)\det(H_f(x))dx \quad \forall f\in\Conv(V,\R)\cap C^2(V), B\subset V \text{ Borel set}, 
		\end{align*}
		where $\mathcal{M}(V)$ is equipped with the vague topology, i.e. the topology induced by the semi-norms $|\mu|_\phi:=|\int_V\phi d\mu|$ for $\phi\in C_c(V)$. Then $\mathcal{M}(V)$ is a complete locally convex vector space that does not admit a continuous norm. The Goodey-Weil distribution $\GWVConv(\Phi_n)[\phi_1\otimes...\otimes\phi_n]$ is the signed measure given by integrating the mixed determinant of the Hessians of the functions $\phi_1,...,\phi_n\in C^\infty_c(V)$. In particular, $\GWVConv(\Phi_n)$ does not have compact support.

	\section{A notion of support for dually epi-translation invariant valuations}
	\label{section_support_of_a_valuation}
	Throughout this section, let $F$ be a locally convex vector space. Motivated by Theorem \ref{theorem_support_GW_diagonal} we make the following definition:
	\begin{definition}
		\label{defition_support_valuation}
		For $1\le k\le n$ and $\mu\in\VConv_k(C;V,F)$ let the support $\supp\mu\subset V$ be the set
		\begin{align*}
			\supp\mu:=\bigcap\limits_{A\subset V\text{ closed, } \supp \GWVConv(\mu)\subset \Delta A}A.
		\end{align*}
		Here $\Delta:V\rightarrow V^k$ is the diagonal embedding. For $\mu\in\VConv_0(C;V,F)$, we set $\supp\mu=\emptyset$. If $\mu=\sum_{i=0}^{n}\mu_i$ is the homogeneous decomposition of $\mu\in \VConv(C;V,F)$ we set $\supp\mu:=\bigcup_{i=0}^n\supp\mu_i$. 
	\end{definition}
	Theorem \ref{theorem_compact_support_GW} implies
	\begin{corollary}
		If $F$ admits a continuous norm, then any $\mu\in\VConv(C;V,F)$ has compact support.
	\end{corollary}
	Let us justify the terminology:
	 \begin{proposition}
	 	\label{proposition_support_convex_valuation}
	 	The support of $\mu\in\VConv(V)$ is minimal (with respect to inclusion) amongst the closed sets $A\subset V$ with the following property: If $f,g\in\Conv(V,\R)$ satisfy $f=g$ on an open neighborhood of $A$, then $\mu(f)=\mu (g)$.
	 \end{proposition}
	 \begin{proof}
	 	Let us first show that any closed set $A$ satisfying the property contains the support of $\mu$. Using the homogeneous decomposition, we can assume that $\mu$ is $k$-homogeneous.\\
	 	We will argue by contradiction. Assume the support was not contained in $A$. Then $ \supp \GWVConv(\mu)\setminus \Delta A\ne \emptyset$. In particular, we find functions $\phi_1,...,\phi_k$ with support in $V\setminus A$ such that
	 	\begin{align*}
		 	\GWVConv(\mu)(\phi_1\otimes...\otimes \phi_k)\ne0.
	 	\end{align*}Using the Hahn-Banach theorem, we can choose $\lambda\in F'$ with \begin{align*}
	 		\lambda(\GWVConv(\mu))(\phi_1\otimes\dots\otimes \phi_k))\ne0. 
	 	\end{align*}
	 	Choose an euclidean structure on $V$ and let $f\in\Conv(V,\R)$ be given by $f(x):=|x|^2$. Then $f+\sum_{i=1}^{k}\delta_i\phi_i$ is convex for all $\delta_i$ small enough. Let us compare $\lambda(\mu(f))$ and $\lambda(\mu(f+\sum_{i=1}^{k}\delta_i\phi_i))$. By construction, the two functions coincide on an open neighborhood of $A$ and thus the assumption implies $\lambda(\mu(f))=\lambda(\mu(f+\sum_{i=1}^{k}\delta_i\phi_i))$  for all $\delta_i$ small enough. Applying Theorem \ref{theorem_compact_support_GW}, we see that the right hand side is a polynomial in $\delta_i$ for all $\delta_i$ small enough and that the coefficient in front of $\delta_1\dots\delta_k$ is exactly $k!\GWVConv(\lambda\circ \mu)(\phi_1\otimes\dots\otimes\phi_k)=k!\lambda(\GWVConv(\mu)(\phi_1\otimes...\otimes\phi_k))$. As the left hand side is independent of $\delta_i$, this coefficient has to vanish, so we obtain a contradiction.\\
	 	
	 	It remains to see that $\supp\mu$ actually satisfies the property. 	Again, we can assume that $\mu$ is $k$-homogeneous. As $F$ is locally convex, it is sufficient to show the claim for all valuations $\lambda\circ\mu\in\VConv(C;V)$ for $\lambda\in F'$. As this is a real-valued valuation, its support is compact, so under the assumptions above the mollified functions $f_\epsilon,g_\epsilon\in \Conv(V,\R)\cap C^\infty(V)$ satisfy $f_\epsilon=g_\epsilon$ on an open neighborhood of the support of $\lambda\circ\mu$ for all $\epsilon>0$ small enough. In particular, $f_\epsilon^{\otimes k}=g_\epsilon^{\otimes k}$ on a neighborhood of $\supp \GWVConv(\lambda\circ\mu)$ and using Theorem \ref{theorem_compact_support_GW} we obtain 
	 	\begin{align*}
	 	\lambda(\mu(f))=&\lim\limits_{\epsilon\rightarrow0}\lambda(\mu(f_\epsilon))=\lim\limits_{\epsilon\rightarrow0}\GWVConv(\lambda\circ\mu)\left(f^{\otimes k}_\epsilon\right)=\lim\limits_{\epsilon\rightarrow0}\GWVConv(\lambda\circ\mu)\left(g^{\otimes k}_\epsilon\right)\\
	 	=&\lim\limits_{\epsilon\rightarrow0}\lambda(\mu(g_\epsilon))=\lambda(\mu(g)).
	 	\end{align*}
	 \end{proof}

	As a first application of this notion of support, we will discuss the (non-) existence of invariant real-valued valuations for non-compact groups $G\subset Aff(V)$, where a valuation $\mu\in\VConv(V)$ is called $G$-invariant, if $\mu(f\circ g)=\mu(f)$ for all $f\in\Conv(V,\R)$ and $g\in G$. We need the following preparatory proposition:
	\begin{proposition}
		\label{proposition_no_concentration_of_support_in_1_point}
		If the support of $\mu\in\VConv(C;V,F)$ is contained in a one-point set, then it is empty and $\mu$ is constant.
	\end{proposition}
	\begin{proof}
		By considering $\lambda\circ\mu$ for $\lambda\in F'$ again, it is enough to consider the case $F=\R$. Let us also assume $V=\R^n$ and, without loss of generality, let the support of $\mu$ be contained in $\{0\}$. By taking the homogeneous decomposition of $\mu$, we can assume that $\mu$ is homogeneous of degree $k$. We thus only need to show that the assumptions imply $\mu=0$ for $k>0$.\\
		If $\mu$ is $1$-homogeneous, $\GWVConv(\mu)$ is a distribution with compact support of order at most $2$ due to Inequality \eqref{equation_continuity_GW}, so there exist constants $c_\alpha\in\R$ such that
		\begin{align*}
		\GWVConv(\mu)=\sum\limits_{|\alpha|\le 2}c_\alpha\partial^\alpha\delta_0.
		\end{align*}
		Plugging in linear and constant functions, we see that $c_\alpha=0$ for $|\alpha|<2$. Thus for any $f\in C^\infty(V)$:
		\begin{align*}
		\GWVConv(\mu)(f)=\sum\limits_{|\alpha|=2}c_\alpha\partial^\alpha f(0).
		\end{align*}
		Fix $1\le i\le n$ and consider the functions $f_\epsilon(x)=\sqrt{\epsilon^2+x_i^2}$ for $\epsilon>0$. Then 
		\begin{align*}
		\partial^\alpha f_\epsilon(x)=\begin{cases}
		\frac{\epsilon^2}{\sqrt{\epsilon^2+x_i^2}^3} & \alpha=(i,i),\\
		0 & \text{else}.
		\end{cases}
		\end{align*}
		Moreover, $f_\epsilon(x)\rightarrow f(x)=|x_i|$ for $\epsilon\rightarrow0$, so the continuity of $\mu$ implies
		\begin{align*}
		\mu(f)=\lim\limits_{\epsilon\rightarrow0}\mu(f_\epsilon)=\lim\limits_{\epsilon\rightarrow0}\GWVConv(\mu)(f_\epsilon)	=\lim\limits_{\epsilon\rightarrow0}c_{(i,i)} \frac{1}{\epsilon}.
		\end{align*}
		Thus we must have $c_{(i,i)}=0$. In total, we are left with an expression of the form 
		\begin{align*}
		\GWVConv(\mu)=\sum\limits_{i< j}c_{ij}\partial_i\partial_j\delta_0.
		\end{align*}
		Now consider $f_\epsilon(x)=\sqrt{\epsilon^2+(x_i+x_j)^2}$ for $i\ne j$, which converges to $f(x)=|x_i+x_j|$ for $\epsilon\rightarrow0$. Then $\partial_i\partial_jf_\epsilon(x)=\frac{\epsilon^2}{\sqrt{\epsilon^2+(x_i+x_j)^2}^3}$ and all other mixed derivatives vanish, so the same argument as before shows that
		\begin{align*}
		\mu(f)=\lim\limits_{\epsilon\rightarrow0}\mu(f_\epsilon)=\lim\limits_{\epsilon\rightarrow0}\GWVConv(\mu)(f_\epsilon)=\lim\limits_{\epsilon\rightarrow0}c_{ij} \frac{1}{\epsilon}.
		\end{align*}
		Thus $c_{ij}=0$ for all $1\le i,j\le n$, i.e. $\GWVConv(\mu)=0$. The injectivity of $\GWVConv$ from Corollary \ref{corollary_GW_injective} implies $\mu=0$.\\
		
		If $\mu$ is $k$-homogeneous, we consider the valuation
		\begin{align*}
		\mu_f:=\bar{\mu}(\cdot,f[k-1])
		\end{align*}
		for $f\in C$ obtained from $\bar{\mu}$ by setting the last $k-1$ arguments equal to $f$. Then $\mu_f$ is a $1$-homogeneous valuation. Using Proposition \ref{proposition_support_convex_valuation}, it is easy to see that the support of $\mu_f$ is a subset of the support of $\mu$, so we deduce $\mu_f=0$ from the case $k=1$. In particular, $\mu(f)=\bar{\mu}(f,f[k-1])=\mu_f(f)=0$.
	\end{proof}
	\begin{corollary}
		Let $G\subset Aff(V)$ be a subgroup such that either
		\begin{enumerate}
			\item there exists no compact orbit in $V$, or
			\item the only compact orbit in $V$ consists of a single point.
		\end{enumerate}
		Then any $G$-invariant valuation in $\VConv(V)$ is constant. In particular, any translation or $\mathrm{SL}(V)$-invariant valuation (for $\dim V\ge 2$) is constant.
	\end{corollary}
	\begin{proof}
		Without loss of generality we can assume that $\mu$ is homogeneous of degree $k$ and $G$-invariant. We will show that $\mu$ has to vanish identically if $k>0$.\\ Suppose $k>0$. It is easy to see that $\GWVConv:\VConv_k(V)\rightarrow\mathcal{D}'(V^k)$ is equivariant with respect to the operation of the affine group. In particular, any $G$-invariant valuation induces a $G$-invariant distribution. As the support of any such distribution must be invariant with respect to the group, the same holds true for the support of $\mu$. However, the support of $\mu$ is compact, so we directly see that the support of $\mu$ is either empty or consists of a single point. Due to Proposition \ref{proposition_no_concentration_of_support_in_1_point} the second case cannot occur, so the support of $\mu$ is empty, i.e. $\mu=0$.
	\end{proof}
	
	\subsection{Subspaces of valuations with compact support}
	\label{section_subspaces_compact_support}
	The goal of this section is to establish some useful results on the topology of spaces of valuations with support contained in a fixed (compact) set. An application of these results will be presented in an upcoming work.\\
	Throughout this section let $F$ be a locally convex vector space and let us assume for simplicity that $V$ carries some euclidean structure. 
	\begin{definition}
		For $A\subset V$ we denote by $\VConv_A(V,F)$ the space of valuations that have support in $A$.
	\end{definition}
	\begin{lemma}
		\label{lemma_spaces_with_fixed_support_closed_in_VConv}
		If $A\subset V$ is closed, then $\VConv_A(V,F)$ is a closed subspace of $\VConv(V,F)$.
	\end{lemma}
	\begin{proof}
		If $(\mu_\alpha)_\alpha$ is a net in $\VConv_A(V,F)$ converging to $\mu$ in $\VConv(V,F)$ and $f,h\in\Conv(V,\R)$ are two functions with $f=h$ on a neighborhood of $A$, we deduce $\mu_\alpha(f)=\mu_\alpha(h)$ for all $\alpha$ using Proposition \ref{proposition_support_convex_valuation}. Taking the limit, we obtain $\mu(f)=\mu(h)$. As this is true for any $f,h\in\Conv(V,\R)$ with $f=h$ on a neighborhood of $A$, the support of $\mu$ has to be contained in $A$ by Proposition \ref{proposition_support_convex_valuation}. Thus $\VConv_A(V,F)$ is closed in $\VConv(V,F)$.
	\end{proof}
	To illustrate the relevance of these spaces, let us prove Proposition \ref{mainproposition_approximation_using_sequences}:
	\begin{proof}[Proof of Proposition \ref{mainproposition_approximation_using_sequences}]
		Let us denote the continuous norm by $\|\cdot\|$ and let $U_R:=U_R(0)$ denote the open ball  in $V$ with radius $R>0$. Using the homogeneous decomposition, we can assume that all valuations are $k$-homogeneous.\\
		Assume that the supports of the valuations $\mu_j$ are not bounded. Choosing a subsequence if necessary, we can assume that the following holds: There exists a strictly increasing sequence $(r_j)_j$ of positive real numbers converging to $+\infty$ such that
		\begin{enumerate}
			\item $\supp\mu\subset U_{r_0}$,
			\item $\supp\mu_j\subset U_{r_j}$ for all $j\ge 1$,
			\item $\supp\mu_{j+1}\setminus B_{r_j}\ne \emptyset$ for all $j\ge 1$.
		\end{enumerate}
		In particular, for every $j\in\mathbb{N}$ we can inductively define functions  $\phi^j_1,...,\phi^j_k\in C^\infty_c(V)$ with the properties
		\begin{enumerate}
			\item $\supp\phi^j_i\subset U_ {r_j}\setminus B_{r_{j-1}}$ for all $j\ge 1$,
			\item $\|\sum\limits_{l=1}^{j}\GWVConv(\mu_j)(\phi_1^l\otimes...\otimes\phi^l_k)\|\ge1$ for all $j\ge 1$, 
		\end{enumerate}
		as follows: Assume that we have constructed the functions $\phi_i^l$ for all $1\le i\le k$ and $l\le j-1$. 
		If $\|\sum_{l=1}^{j-1}\GWVConv(\mu_j)(\phi_1^l\otimes...\otimes\phi^l_k)\|\ge1$, choose $\phi_1^j=...=\phi_k^j=0$. \\
		If $\|\sum_{l=1}^{j-1}\GWVConv(\mu_j)(\phi_1^l\otimes...\otimes\phi^l_k)\|< 1$, choose $\phi_i^j\in C_c^\infty(U_{r_j}\setminus B_{j-1})$ such that $\GWVConv(\mu_j)(\phi_1^j\otimes\dots\otimes\phi_k^j)\ne 0$. Then
		\begin{align*}
		\|\sum\limits_{l=1}^{j}\GWVConv(\mu_j)(\phi_1^l\otimes...\otimes\phi^l_k)\|
		\ge& \|\GWVConv(\mu_j)(\phi_1^j\otimes...\otimes\phi^j_k)\|-\|\sum\limits_{l=1}^{j-1}\GWVConv(\mu_j)(\phi_1^l\otimes...\otimes\phi^l_k)\|\\
		>&\|\GWVConv(\mu_j)(\phi_1^j\otimes...\otimes\phi^j_k)\|-1.
		\end{align*}
		Scaling one of the functions $\phi_i^j$ appropriately for $1\le i\le k$, we can make the right hand side equal to $1$.\\
		In any case, we obtain functions satisfying $\|\sum_{l=1}^{j}\GWVConv(\mu_j)(\phi_1^l\otimes...\otimes\phi^l_k)\|\ge 1$ for all $j\ge 1$.\\
		
		For $1\le i\le k$ define $\phi_i:=\sum_{j=1}^{\infty}\phi_i^j$. By construction, this is a locally finite sum, so we obtain an element in $C^\infty(V)$. As the supports of the functions $(\phi_i^j)_j$ are  pairwise disjoint for each $1\le i\le k$, we can apply Lemma \ref{lemma_smooth function difference of convex functions} to find functions $f_i\in\Conv(V,\R)$, $1\le i\le k$, such that  $f_i^j:=f_i+\sum_{l=1}^j\phi_i^l$ is convex for all $1\le i\le k$, $j\in\mathbb{N}$. Then $(f_i^j)_j$ converges to $f_i+\phi_i$ uniformly on compact subsets, i.e. in $\Conv(V,\R)$. Furthermore, $f_i^j=f_i$ on an open neighborhood of the support of $\mu$, so $\mu(f_i)=\mu(f^j_i)$ for all $j$. As the polarization $\bar{\mu}$ is a linear combination of $\mu$ evaluated in positive linear combinations of the arguments, exchanging $f_i$ and $f_i^j$ does not change the value of $\bar{\mu}$. For any $j\in\mathbb{N}$  we thus obtain		
		\begin{align*}
		0=&\|\GWVConv(\mu)(0\otimes...\otimes0)\|\\
		=&\|\sum\limits_{i=0}^k(-1)^{k-i}\frac{1}{(k-i)!i!}\sum_{\sigma\in S_k}\bar{\mu}\left(f_{\sigma(1)},...,f_{\sigma(i)},f_{\sigma(i+1)},...,f_{\sigma(k)}\right)\|\\
		=&\|\sum\limits_{i=0}^k(-1)^{k-i}\frac{1}{(k-i)!i!}\sum_{\sigma\in S_k}\bar{\mu}\left(f^j_{\sigma(1)},...,f^j_{\sigma(i)},f_{\sigma(i+1)},...,f_{\sigma(k)}\right)\|,
		\end{align*}
		i.e. $\sum_{i=0}^k(-1)^{k-i}\frac{1}{(k-i)!i!}\sum_{\sigma\in S_k}\bar{\mu}(f^j_{\sigma(1)},...,f^j_{\sigma(i)},f_{\sigma(i+1)},...,f_{\sigma(k)})=0$.\\
		Set $K:=\{f_i^j \ : \ j\in\mathbb{N},1\le i\le k\}\cup\{f_1+\phi_1,...,f_k+\phi_k,f_1,...,f_k\}$. Then $K\subset\Conv(V,\R)$ is compact, so $(\mu_j)_j$ converges to $\mu$ uniformly on $K$. By Lemma \ref{lemma_continuity_valuation->polarization} the same holds for the polarizations $(\bar{\mu}_j)_j$. In particular, there exists $N\in\mathbb{N}$ such that 
		\begin{align*}
		&\|\sum\limits_{i=0}^k(-1)^{k-i}\frac{1}{(k-i)!i!}\sum_{\sigma\in S_k}\bar{\mu}_j(f^j_{\sigma(1)},...,f^j_{\sigma(i)},f_{\sigma(i+1)},...,f_{\sigma(k)})\|\\
		=&\|\sum\limits_{i=0}^k(-1)^{k-i}\frac{1}{(k-i)!i!}\sum_{\sigma\in S_k}\bar{\mu}(f^j_{\sigma(1)},...,f^j_{\sigma(i)},f_{\sigma(i+1)},...,f_{\sigma(k)})\\
		&-\sum\limits_{i=0}^k(-1)^{k-i}\frac{1}{(k-i)!i!}\sum_{\sigma\in S_k}\bar{\mu}_j(f^j_{\sigma(1)},...,f^j_{\sigma(i)},f_{\sigma(i+1)},...,f_{\sigma(k)})\|<\frac{1}{2}
		\end{align*}
		for all $j\ge N$. By definition
		\begin{align*}
		&\sum\limits_{i=0}^k(-1)^{k-i}\frac{1}{(k-i)!i!}\sum_{\sigma\in  S_k}\bar{\mu}_j(f^j_{\sigma(1)},...,f^j_{\sigma(i)},f_{\sigma(i+1)},...,f_{\sigma(k)})\\
		=& \GWVConv(\mu_j)\left(\sum_{l=1}^j\phi_1^l\otimes...\otimes \sum_{l=1}^j\phi_k^l\right).
		\end{align*}
		As the support of $\GWVConv(\mu_j)$ is contained in the diagonal and the functions belonging to different superscripts $i$ have disjoint support, we obtain
		\begin{align*}
		&\sum\limits_{i=0}^k(-1)^{k-i}\frac{1}{(k-i)!i!}\sum_{\sigma\in S_k}\bar{\mu}_j(f^j_{\sigma(1)},...,f^j_{\sigma(i)},f_{\sigma(i+1)},...,f_{\sigma(k)})=\sum\limits_{l=1}^{j}\GWVConv(\mu_j)(\phi_1^l\otimes...\otimes \phi_k^l).
		\end{align*}
		Thus we arrive at 
		\begin{align*}
		\|\sum\limits_{l=1}^{j}\GWVConv(\mu_j)(\phi_1^l\otimes...\otimes \phi_k^l)\|<\frac{1}{2}
		\end{align*}
		for all $j\ge N$, which is a contradiction.
	\end{proof}
	
	In the rest of this section, we introduce special continuous semi-norms on $\VConv_A(V,F)$ for compact subsets $A\subset V$. The main goal for the introduction of these semi-norms is to simplify convergence arguments. Together with Proposition \ref{mainproposition_approximation_using_sequences} this gives us a rather effective set of tools for approximation problems.
	\begin{proposition}
		\label{proposition_family_of_continuous_norms_on_VConv}
		Let $A\subset V$ be compact and convex. Let $|\cdot|_F$ denote a continuous semi-norm on $F$ and choose $s>0$. For $\mu\in \VConv_{A}(V,F)$ define
		\begin{align*}
		\|\mu\|_{F;A,s}:=\sup\{|\mu(f)|  \ : \ f\in\Conv(V,\R), \|f\|_{C(A+2s B_1)}\le 1 \}.
		\end{align*}
		This defines a continuous semi-norm on $\VConv_{A}(V)$. If $|\cdot|_F$ is a norm, so is $\|\cdot\|_{F;A,s}$. In addition, the topology induced by the family $\|\cdot\|_{F;A,s}$ (for all continuous semi-norms $|\cdot|_F$ on $F$) on $\VConv_{A}(V,F)$ coincides with the relative topology. 
	\end{proposition}
	\begin{proof}
		It is clear that $\|\cdot\|_{F;A,s}$ defines a semi-norm if it is finite. Let $f\in\Conv(V,\R)$ with $\|f\|_{C(A+2sB_1)}\le 1$ be given. By Proposition \ref{proposition_convex_functions_local_lipschitz_constants}, $f$ is Lipschitz continuous on $B_{A+sB_1}$ with Lipschitz constant $L=\frac{2}{s}\|f|_{A+2sB_1}\|_{\infty}\le \frac{2}{s}$. Consider the function
		\begin{align*}
		\tilde{f}(x):=\begin{cases}
		\sup\limits_{x=\lambda y+(1-\lambda)z,\lambda\ge1}\lambda f(y)+(1-\lambda)f(z) & x\in V\setminus (A+sB_1),\\
		f(x) & x\in A+sB_1.
		\end{cases}
		\end{align*}
		By the proof of Theorem 4.1 in \cite{Yan:Extension_convex_function}, $\tilde{f}$ is a finite-valued convex extension of the Lipschitz continuous function $f|_{A+sB_1}$. For any $\lambda\ge 1$, $y,z\in A+sB_1$ with $x=\lambda y+(1-\lambda)z$:
		\begin{align*}
		\lambda f(y)+(1-\lambda)f(z)\le&|\lambda[f(y)-f(z)]|+|f(z)|\le \frac{2}{s}\lambda |y-z|+\|f\|_{C(A+sB_1)}\\
		\le& \frac{2}{s}|\lambda y-\lambda z|+1=\frac{2}{s}|x-z|+1.
		\end{align*}
		For $x\in V\setminus (A+sB_1)$ we thus obtain
		\begin{align*}
		\tilde{f}(x)\le \frac{2}{s} \sup_{z\in A+sB_1}|x-z|+1\le \frac{2}{s} (\mathrm{dist}(x,A+sB_1)+\diam(A+sB_1))+1.
		\end{align*}
		Choosing $\lambda=\frac{|z-x|}{s}$, $y=z+s\frac{x-z}{|z-x|}$ and $z\in A$, we also obtain the inequality
		\begin{align*}
		\frac{|z-x|}{s}f(z+s\frac{x-z}{|z-x|})+(1-\frac{|z-x|}{s})f(z)\le\tilde{f}(x).
		\end{align*}
		As 
		\begin{align*}
		&|\frac{|z-x|}{s}f(z+s\frac{x-z}{|z-x|})+(1-\frac{|z-x|}{s})f(z)|\le \frac{|z-x|}{s}+|(1-\frac{|z-x|}{s})|\\
		\le& 2\frac{|z-x|}{s}+1\le \frac{2}{s} (\mathrm{dist}(x,A+sB_1)+\diam(A+sB_1))+1,
		\end{align*}
		$|\tilde{f}(x)|\le \frac{2}{s} (\mathrm{dist}(x,A+sB_1)+\diam(A+sB_1))+1$ for all $x\in V$, so the set 
		\begin{equation*}
		K:=\{f\in\Conv(V,\R) \ : \ f=\tilde{h} \text{ for some }h\in\Conv(V,\R) \text{ with }\ \|h\|_{C(A+2sB_1)}\le 1\}
		\end{equation*}
		is uniformly bounded on compact subsets and therefore relatively compact in $\Conv(V,\R)$ due to Proposition \ref{proposition_compactness_Conv}. In particular, $\mu$ is bounded on $K$, as it is continuous.\\
		Any function $f\in \Conv(V,\R)$ satisfies $\tilde{f}=f$ on $A+sB_1$, i.e. these functions coincide on an open neighborhood of the support of $\mu$. Proposition \ref{proposition_support_convex_valuation} implies $\mu(f)=\mu(\tilde{f})$, and therefore
		\begin{align*}
		\|\mu\|_{F;A,s}=\sup\{\left|\mu(f)\right|_F  \ : \ f\in\Conv(V,\R), \|f\|_{A+sB_1}\le 1 \}=\sup_{\tilde{f}\in K}|\mu(\tilde{f})|_F<\infty.
		\end{align*}
		In addition, we see that the compact subset $\bar{K}\subset\Conv(V,\R)$ satisfies
		\begin{align*}
		\|\mu\|_{F;A,s}\le \|\mu\|_{F;\bar{K}} \quad \text{for all }\mu\in \VConv_{A}(V,F).
		\end{align*}
		On the other hand, any $f\in \bar{K}$ satisfies $\|f\|_{C(A+2sB_1)}\le \sup_{x\in A+2sB_1}\frac{2}{s} (\mathrm{dist}(x,A+sB_1)+\diam(A+sB_1))+1\le c_{A,s}:=\frac{2}{s} (\diam(A)+3s)+1=\frac{2}{s} \diam(A)+7$. By considering the $k$-homogeneous component $\mu_k$ of $\mu$, we obtain
		\begin{align*}
		\|\mu_k\|_{F;\bar{K}}=\sup_{f\in\bar{K}}\left|\mu_k(f)\right|=c_{A,s}^k\sup_{f\in\bar{K}}\left|\mu_k\left(\frac{f}{c_{A,s}}\right)\right| \le c_{A,s}^k\|\mu_k\|_{F;A,s}.
		\end{align*}
		Thus $\|\cdot\|_{F;A,s}$ and $\|\cdot\|_{F;\bar{K}}$ are equivalent, so the semi-norm $\|\cdot\|_{F;A,s}$ is in particular continuous on $\VConv_A(C;V,F)$.\\
		More generally, any compact set $D\subset \Conv(V,\R)$ satisfies $t:=\sup_{f\in D,x\in A+2sB_1}|f(x)|<\infty$. Assuming $t>0$, this implies
		\begin{align*}
		\|\mu_k\|_{F;D}=\sup_{f\in D}\left|\mu_k(f)\right|_F=t^k\sup_{f\in D}\left|\mu_k\left(\frac{f}{t}\right)\right|_F \le t^k\|\mu_k\|_{F;A,s}.
		\end{align*}
		If $t=0$, then any $f\in D$ coincides with the zero function on a neighborhood of the support of $\mu$, so $\mu_k(f)=\mu_k(0)$ for all $f\in D$ due to Proposition \ref{proposition_support_convex_valuation}, i.e. $\|\mu_k\|_{F;D}\le \|\mu_k\|_{F;A,s}$. \\
		In any case, we see that $\|\cdot\|_{F;A,s}$ defines a continuous semi-norm on $\VConv_{A}(V,F)$ and that the family of these semi-norms generates the subspace topology.\\
		Let us now assume that $|\cdot|_F$ is a norm. If $\mu\ne0$, we can find $f\in \Conv(V,\R)$ with $\mu(f)\ne 0$. Repeating the argument above for $D=\{f\}$, we see that $\|\mu\|_{F;D}>0$ for $\mu\in\VConv_A(V,F)$ implies $\|\mu\|_{F;A,s}>0$. Thus $\|\cdot\|_{F;A,s}$ is indeed a norm.
	\end{proof}
	For completeness, let us relate these semi-norms for different parameters $s>0$:
	\begin{corollary}
		Let $A\subset V$ be a compact convex subset. For $0<s<t$
		\begin{align*}
		\|\mu\|_{F;A,t}\le\|\mu\|_{F;A,s}\le \left(\frac{2}{s} (2t+\diam A)+1\right)^k\|\mu\|_{F;A,t}
		\end{align*}
		for all $k$-homogeneous $\mu\in\VConv_A(V,F)$.
	\end{corollary}
	\begin{proof}
		The first inequality is obvious. For the second inequality, let $f\in\Conv(V,\R)$ be a function with $\|f\|_{C(A+2sB_1)}\le 1$. Considering the function $\tilde{f}\in\Conv(V,\R)$ given by
		\begin{align*}
		\tilde{f}(x):=\begin{cases}
		\sup\limits_{x=\lambda y+(1-\lambda)z,\lambda\ge1}\lambda f(y)+(1-\lambda)f(z) & x\in V\setminus (A+sB_1),\\
		f(x) & x\in A+sB_1
		\end{cases}
		\end{align*}
		from the previous proof, we see that $|\tilde{f}(x)|\le \frac{2}{s} (\mathrm{dist}(x,A+sB_1)+\diam(A+sB_1))+1$, so $\|\tilde{f}\|_{C(A+2tB_1)}\le \frac{2}{s} (2t-s +\diam A+s)+1=\frac{2}{s} (2t+\diam A)+1$. As $f=\tilde{f}$ on a neighborhood of the support of $\mu$, we obtain
		\begin{align*}
		|\mu(f)|_F&=\left(\frac{2}{s} (2t+\diam A)+1\right)^k\left|\mu\left(\frac{f}{\frac{2}{s} (2t+\diam A)+1}\right)\right|_F\\
		&\le \left(\frac{2}{s} (2t+\diam A)+1\right)^k\|\mu\|_{F;A,s}.
		\end{align*}
	\end{proof}
	\begin{corollary}
		\label{corollary_Frechet-topology_for_compactly_supported_valuations}
		If $A$ is compact and $F$ is a Banach or Fr\'echet space, then $\VConv_A(V,F)$ is also a Banach or Fr\'echet space respectively.
	\end{corollary}
	\begin{proof}
		By Lemma \ref{lemma_spaces_with_fixed_support_closed_in_VConv}, $\VConv_A(V,F)$ is a closed subspace of the complete locally convex space $\VConv(V,F)$ and so it is also complete.\\
		If $A$ is compact and convex, we can take one of the semi-norms from Proposition \ref{proposition_family_of_continuous_norms_on_VConv}, which generates the subspace topology, so the space $\VConv_A(V,F)$ is complete with respect to this semi-norms. If $F$ is a Banach space, we only obtain one norm, while we get a sequence of norms if $F$ is a Fr\'echet space. In both cases, the claim follows\\
		If $A$ is not convex, choose $R>0$ such that $A\subset B_R(0)$. Using the same argument as in Lemma \ref{lemma_spaces_with_fixed_support_closed_in_VConv}, we see that $\VConv_A(V,F)\subset\VConv_{B_R(0)}(V,F)$ is a closed subspace of a Banach or Fr\'echet space. The claim follows.
	\end{proof}
	\subsection{Vertical support of valuations on convex bodies and the image of the embedding $T:\VConv(V,F)\rightarrow \Val(V^*\times\R,F)$}
	\label{section_vertical_support_and_image_embedding}
		Similar to the definition of support of valuations in $\VConv(C;V,F)$, we will give a notion of \emph{vertical support} for elements of $\Val(V)$. Starting point is the Goodey-Weil embedding for translation invariant valuations on convex bodies. Consider the space $\mathbb{P}_+(V^*)$ of oriented lines in $V^*$ and the line bundle $L$ over $\mathbb{P}_+(V^*)$ with fiber over $l\in\mathbb{P}_+(V^*)$ given by
		\begin{align*}
			P_l:=\{h:l^+\rightarrow\R \ 1\text{-homogeneous}\}.
		\end{align*}
		Note that every support function defines a continuous section of $L$. For $y\in V^*\setminus\{0\}$ we will write $[y]$ for the corresponding oriented line in $\mathbb{P}_+(V^*)$.
		\begin{theorem}
			Let $F$ be a locally convex vector space. For every $\mu\in\Val_k(V,F)$ there exists a unique distribution $\GW(\mu)\in \mathcal{D}'(\mathbb{P}_+(V^*)^k,L^{\boxtimes k}),\bar{F})$ such that
			\begin{align*}
				\GW(\mu)(h_{K_1}\otimes\dots\otimes h_{K_k})=\bar{\mu}(K_1,\dots,K_k).
			\end{align*}
			Here $\bar{\mu}$ denotes the polarization of $\mu\in\Val_k(V,F)$. Furthermore, the support of this distribution is contained in the diagonal.
		\end{theorem}
		\begin{proof}
			In \cite{Goodey_Weil:Distributions_and_valuations}, the existence of such a distribution was shown in the case $F=\R$. The same construction can be done for arbitrary locally convex vector spaces $F$, similar to our construction of the Goodey-Weil embedding for $\VConv(V,F)$. The diagonality of the support was first shown in \cite{Alesker:McMullenconjecture} for real-valued valuations, but as in the proof of Theorem \ref{theorem_support_GW_diagonal}, this implies the more general statement.
		\end{proof}
		Following the approach in the previous section, we define the \emph{vertical support}:
		\begin{definition}
			For $1\le k\le n$, we define the vertical support of $\mu\in\Val_k(V,F)$ to be the set 
			\begin{align*}
			\vsupp \mu:=\bigcap\limits_{\supp \GW(\mu)\subset \Delta A,\ A\subset \mathbb{P}_+(V^*)\text{ compact}} A.
			\end{align*}
			For $k=0$, we set $\vsupp\mu=\emptyset$. If $\mu=\sum_{i=0}^{n}\mu_i$ is the homogeneous decomposition, we set $\vsupp\mu:=\bigcup_{i=0}^n\vsupp\mu_i$.
		\end{definition}
		As before, the vertical support can be characterized without reference to the Goodey-Weil embedding.
		\begin{lemma}
			\label{lemma_property_vertical_support}
			If $K,L\in\mathcal{K}(V)$ are two convex bodies with $h_K=h_L$ on an open neighborhood of $\vsupp \mu$, then $\mu(K)=\mu(L)$.
		\end{lemma}
		\begin{proof}
			Using the homogeneous decomposition, we can assume that $\mu$ is $k$-homogeneous.\\
			Let us choose identify $V\cong \R^n$ to identify $\mathbb{P}_+(V^*)\cong S(V)$, which also trivializes $L$. Take a sequence of positive functions $\phi_j\in C^\infty(SO(n))$ with $\int_{SO(n)}\phi_j(g)dg=1$, such that the diameter of the support of $\phi_j$ converges to zero for $j\rightarrow\infty$. It is easy to see that $f_j(v):=\int_{SO(n)}\phi_j(g)f(g^{-1}v)dg$ defines a sequence of smooth functions on $S(V)$ for every $f\in C(S(V))$, that converges uniformly to $f$. Moreover, if $f=h_K$, then $f_j$ is the restriction of a support function of some convex body $K_j$. The uniform convergence $h_{K_j}\rightarrow h_K$ on $S(V)$ implies that $K_j\rightarrow K$ in the Hausdorff metric. Similarly, we obtain convex bodies $L_j$ from $h_L$.\\
			Note that $f_j(v)$ only depends on the values of $f$ in a neighborhood of $v$ depending on the diameter of the support of $\phi_j$. As the diameter of $\phi_j$ converges to zero and $h_K=h_L$ on a neighborhood of $A$, we see that there exists $N\in\mathbb{N}$ such that  $h_{K_j}=h_{L_j}$ on a neighborhood of $A$ for all $j\ge N$. Thus $h_{K_j}^{\otimes k}=h_{L_j}^{\otimes k}$ on a neighborhood of the support of $\GW(\mu)$ and we obtain
			\begin{align*}
			\mu(K)=\lim\limits_{j\rightarrow\infty}\mu(K_j)=\lim\limits_{j\rightarrow\infty}\GW(\mu)\left(h_{K_j}^{\otimes k}\right)=\lim\limits_{j\rightarrow\infty}\GW(\mu)\left(h_{L_j}^{\otimes k}\right)=\lim\limits_{j\rightarrow\infty}\mu(L_j)=\mu(L).
			\end{align*}  
		\end{proof}
		\begin{proposition}
			\label{proposition_characterization_support_Goodey_Weil_for_Val}
			Let $\mu\in \Val(V)$ and $A\subset \mathbb{P}_+(V^*)$ be a compact subset with the following property: If $K,L\in\mathcal{K}(V)$ are two convex functions with $h_K=h_L$ on an open neighborhood of $A$, then $\mu(K)=\mu (L)$. Then the vertical support of $\mu$ is contained in $A$.
		\end{proposition}
		\begin{proof}
			Again let us choose a metric on $V$ and identify $\mathbb{P}_+(V^*)\cong S(V)$, trivializing $L$.\\
			Using the homogeneous decomposition, we can assume that $\mu$ is $k$-homogeneous. Now assume that the claim was false. Then we would find functions $\phi_1,...,\phi_k\in C^\infty(S(V))$ with support contained in $S(V)\setminus A$ such that $\GW(\mu)(\phi_1\otimes...\otimes\phi_k)=1$. Consider the function $1+\sum_{i=1}^{k}\delta_i\phi_i$ on $S(V)$. For small $\delta_i>0$, it is the support function of a convex body $K_\delta$ and by definition, $h_{K_\delta}=1=h_B$ on a neighborhood of $A$, where $B$ is the unit ball in $V$, so by assumption, $\mu(K_\delta)=\mu(B)$. Note that $\mu(K_\delta)$ is a polynomial in $\delta_i>0$. The coefficient before $\delta_1...\delta_k$ is exactly $k!\GW(\mu)(\phi_1\otimes....\otimes \phi_k)=k!$, while the right side does not depend on $\delta_i>0$, i.e. the coefficient has to vanish. Thus we obtain a contradiction.
		\end{proof}
		For $A\subset \mathbb{P}_+(V^*)$ let $\Val_{A}(V)$ denote the subspace of  valuations with vertical support contained in $A$.
		\begin{corollary}
			\label{corollary_Val_subspaces_compact_support_are_Banach}
			Let $A\subset \mathbb{P}_+(V^*)$ be closed. Then $\Val_{A}(V,F)$ is closed in $\Val(V,F)$.
		\end{corollary}
		\begin{proof}
			 As in Lemma \ref{lemma_spaces_with_fixed_support_closed_in_VConv}.
		\end{proof}
	We are now able to describe the image of $T:\VConv(V,F)\rightarrow\Val(V^*\times\R,F)$ in the case, that $F$ admits a continuous norm. Note that by Theorem \ref{maintheorem_GW}, all valuations $\mu\in\VConv(V,F)$ have compact support in this case. We start with the following observation:
	\begin{proposition}
		\label{proposition_compatibility_supports_under_embedding_VCONV-VAL}
		For $\mu\in\VConv(V)$, $\vsupp(T(\mu))\subset P(\supp \mu)$, where 
		\begin{align*}
		P:V\rightarrow& \mathbb{P}_+(V\times\R)\\
		v\mapsto& [(v,-1)].
		\end{align*}
	\end{proposition}	
	\begin{proof}
		By Proposition \ref{proposition_characterization_support_Goodey_Weil_for_Val}, we only need to show that $T(\mu)[K]=T(\mu)[L]$ whenever $h_K$ and $h_L$ coincide on an open neighborhood of $P(\supp \mu)$. Considering $h_K$ and $h_L$ as $1$-homogeneous functions on $V\times\R$, the equality $h_K=h_L$ on an open neighborhood $U$ of $P(\supp \mu)$ implies that they coincide on the open set $\pi^{-1}(U)\subset V\times\R$, where $\pi:(V\times\R)\setminus\{0\}\rightarrow \mathbb{P}_+(V\times\R)$ is the natural projection. Obviously, this is an open neighborhood of $\supp\mu\times\{-1\}$, so we can apply Proposition \ref{proposition_support_convex_valuation}, to obtain $\mu(h_K(\cdot,-1))=\mu(h_L(\cdot,-1))$, i.e. $T(\mu)(K)=T(\mu)(L)$. The claim follows.
	\end{proof}
	\begin{theorem}
		\label{theorem_image_VCONV->Val}
		Let $F$ be a locally convex vector space that admits a continuous norm. The image of $T:\VConv_k(V,F)\rightarrow\Val_k(V^*\times\R,F)$ consists precisely of the valuations $\mu\in\Val_k(V^*\times\R,F)$ whose vertical support is contained in the negative half sphere $\mathbb{P}_+(V\times\R)_-:=\{[(y,s)]\in \mathbb{P}_+(V\times\R):s<0\}$. If $F$ is a Fr\'echet space, $T:\VConv_A(V,F)\rightarrow\Val_{P(A)}(V^*\times\R,F)$ is a topological isomorphism for any compact subset $A\subset V$.
	\end{theorem}
	\begin{proof}
		Starting with $\mu\in\VConv_k(V,F)$, Proposition \ref{proposition_compatibility_supports_under_embedding_VCONV-VAL} shows that $T(\mu)$ has vertical support contained in $\mathbb{P}_+(V\times\R)_-$.\\
		Conversely, let $\nu\in\Val_k(V^*\times\R,F)$ be a valuation with vertical support contained in $\mathbb{P}_+(V\times\R)_-$. As $P:V\rightarrow \mathbb{P}_+(V\times\R)_-$ is a diffeomorphism, $P^{-1}(\vsupp\nu)$ is compact. \\
		Let us construct a functional $\mu$ on $\Conv(V,\R)$ as follows: Given $f\in\Conv(V,\R)$, let $K_f\in\mathcal{K}(V^*\times\R)$ be a convex body with $h_{K_f}(\cdot,-1)=f$ on some neighborhood of $P^{-1}(\vsupp\nu)$, which exists by Proposition \ref{proposition_epi_graph_support_function_convex_body}.
		Now set
		\begin{align*}
		\mu(f):=\nu(K_f).
		\end{align*}
		Note that this does not depend on the special choice of $K_f$: If $K$ is another convex body with $h_K(\cdot,-1)=f$ on some neighborhood of $P^{-1}(\vsupp\nu)$, then $h_K(\cdot,-1)=f=h_{K_f}(\cdot,-1)$ on a neighborhood of $P^{-1}(\vsupp\nu)$, i.e. $h_K=h_{K_f}$ on a neighborhood of $\vsupp\nu$, so Lemma \ref{lemma_property_vertical_support} implies $\nu(K)=\nu(K_f)$. \\
		The functional constructed this way is also a valuation: Choose a scalar product on $V$ and let $R>0$ be such that $P^{-1}(\vsupp\nu)$ is contained in $B_R$. If $\min(f,h)$ is convex, then 
		\begin{align*}
		& \epi \max(f,h)^*=\epi \min(f^*,h^*)=\epi(f^*)\cup \epi(h^*),\\
		& \epi \min(f,h)^*=\epi \max(f^*,h^*)=\epi(f^*)\cap \epi(h^*).
		\end{align*}
		For $c=\max \{||f||_{C(B_{R+2})},||h||_{C(B_{R+2})},||\min \{f,h\}||_{C(B_{R+2})},||\max\{f,h\}||_{C(B_{R+2})}\}$ choose
		\begin{align*}
		K_f&=\epi(f^*)\cap \{|y|\le 2(n+2)c,|t|\le 3(n+2)c\},\\
		K_h&=\epi(h^*)\cap \{|y|\le 2(n+2)c,|t|\le 3(n+2)c\}.
		\end{align*}
		Proposition \ref{proposition_epi_graph_support_function_convex_body} shows that
		\begin{align*}
		&\max(f,h)= h_{K_f\cup K_h}(\cdot,-1), &&\min(f,h)=h_{K_f\cap K_h}(\cdot,-1)\quad \text{on } B_{R+1},
		\end{align*}
		so the definition of $\mu$ implies
		\begin{align*}
		\mu(\max(f,h))+\mu(\min(f,h))=&\nu(K_f\cup K_h)+\nu(K_f\cap K_h)\\
		=&\nu(K_f)+\nu(K_h)=\mu(f)+\mu(h).
		\end{align*}
		Furthermore, $\mu$ is invariant under the addition of linear or constant functions, as $\nu$ is translation invariant. It remains to show that $\mu$ is continuous. We will argue by contradiction.\\
		Let $(f_j)_j$ be a sequence in $\Conv(V,\R)$ converging to $f\in\Conv(V,\R)$ uniformly on compact subsets and assume that there exists $\epsilon>0$ such that $|\mu(f_j)-\mu(f)|>\epsilon$ for all $j\in\mathbb{N}$ for some continuous semi-norm $|\cdot|$ on $F$. Recall that we have chosen $R>0$ such that $P^{-1}(\vsupp\nu)\subset B_R$. As the set $\{f_j|j\in\mathbb{N}\}\cup\{f\}$ is compact, these functions are bounded on $B_{R+2}$ by some constant $c>0$. Using Proposition \ref{proposition_epi_graph_support_function_convex_body}, we see that the convex bodies 
		\begin{align*}
		K_{f_j}&=\epi(f_j^*)\cap \{|y|\le 2(R+2)c,|t|\le 3(R+2)c\},\\
		K_f&=\epi(f^*)\cap \{|y|\le 2(R+2)c,|t|\le 3(R+2)c\},
		\end{align*}
		satisfy $h_{K_{f_j}}=f_j$ and $h_{K_f}=f$ on $B_{R+1}$. By construction, the sequence $(K_{f_j})_j$ of convex bodies is bounded, so by the Blaschke selection theorem we find a subsequence $K_{f_{j_k}}$ converging to some convex body $K\in\mathcal{K}(V^*\times\R,F)$. Then $h_K(\cdot,-1)=h_{K_f}(\cdot,-1)$ on $B_{R+1}$, as $h_{K_{f_{j_k}}}(\cdot,-1)=f_{j_k}$ on $B_{R+1}$ and $f_{j}\rightarrow f$. As $\mu(f)$ does not depend on the special choice of the convex body, we deduce that
		\begin{align*}
		\lim\limits_{k\rightarrow\infty}\mu(f_{j_k})=\lim\limits_{k\rightarrow\infty}\nu(K_{f_{j_k}})=\nu(K)=\nu(K_f)=\mu(f).
		\end{align*}
		This is a contradiction to $|\mu(f_j)-\mu(f)|>\epsilon$ for all $j\in\mathbb{N}$. Thus $\mu$ has to be continuous.\\
		We have constructed $\mu\in\VConv(V,F)$ with $T(\mu)=\nu$ and the support of $\mu$ is obviously contained in $P^{-1}(\vsupp\nu)$.\\
		
		Now let $A\subset V$ be compact, $F$ a Fr\'echet space. Observe that the restriction $T:\VConv_A(V,F)\rightarrow\Val_{P(A)}(V^*\times\R)$ is a well defined, injective, and continuous map between Fr\'echet spaces by Corollary \ref{corollary_Frechet-topology_for_compactly_supported_valuations} and Theorem \ref{theorem_embedding VConv->Val(VxR)}. By the preceding discussion it is also surjective, so Banach's inversion theorem implies that $T^{-1}:\Val_{P(A)}(V^*\times\R,F)\rightarrow\VConv_A(V,F)$ is continuous, i.e. $T:\VConv_A(V,F)\rightarrow\Val_{P(A)}(V^*\times\R,F)$ is a topological isomorphism.
	\end{proof}	
	Note that Proposition \ref{mainproposition_approximation_using_sequences} shows that the inverse $T^{-1}:\mathrm{Im} T\rightarrow \VConv(V;F)$ is not continuous if $F$ admits a continuous norm: If $(\mu_j)_j$ is a sequence in $\Val(V^*\times\R,F)$ that converges to zero such that $\mu_j\in \mathrm{Im} T$ and such that the distance of the vertical supports of these valuations to the set $\{[(v,s)]\in\mathbb{P}_+(V\times\R): s=0\}$ converges to zero, then $T^{-1}(\mu_j)$ defines a sequence of valuations in $\VConv(V,F)$ with unbounded supports. Thus the sequence cannot converge in $\VConv(V,F)$.

\section{Valuations on special cones of convex functions}
	\subsection{Restrictions on the support}
	\label{section_description_valuations_on_cones_with_support}
	In this section we are going to relate valuations on a regular cone $C\subset \Conv(V)$ to valuations on finite-valued convex functions that satisfy certain restrictions on their support. As usual, let $F$ denote a locally convex vector space.
	\begin{theorem}
		\label{maintheorem_support_valuations_non-finite_functions}
		Let $C\subset\Conv(V)$ be a regular cone containing $\Conv(V,\R)$. Consider the sets $B_C:=\bigcap\limits_{f\in C}\overline{\dom f}$, $U_C:=\mathrm{int}B_C$. Then the following holds:
		\begin{enumerate}
			\item The support of any valuation in $\VConv(C;V,F)$ is contained in $B_C$. 
			\item If $F$ admits a continuous norm, then every valuation in $\VConv(V,F)$ with support contained in $U_C$ extends uniquely to an element of $\VConv(C;V,F)$.
		\end{enumerate}
		If $F$ admits a continuous norm, we thus have inclusions 
		\begin{align*}
		\VConv_{U_C}(V,F)\hookrightarrow\VConv(C;V,F)\hookrightarrow \VConv_{B_C}(V,F).
		\end{align*}
	\end{theorem}
	\begin{proof}
		For the first statement, consider the Goodey-Weil distribution of $\mu$ and let $\phi_1,\dots,\phi_k\in C^\infty_c(V\setminus B_C)$. We need to show $\GWVConv(\mu)(\phi_1\otimes \dots\otimes \phi_k)=0$. Using a partition of unity, we can assume that $\supp\phi_i\subset U_\epsilon(x_i)$ for some $x_i\in V\setminus B_C$ and that $B_{\epsilon}(x_i)\subset V\setminus B_C$. We claim that every point $y\in B_\epsilon(x_i)$ has a neighborhood where some $f_{x_i}\in C$ is identical to $\infty$. Indeed, if $y\in B_\epsilon(x_i)$ is a point where the assertion is violated, then $y\in \overline{\dom f}$ for all $f\in C$. Thus $y\in B_C$, which is a contradiction to $y\in B_{\epsilon}(x_i)\subset V\setminus B_C$.  As $B_\epsilon(x_i)$ is compact, we can thus find a finite number of functions $f_{1,i},...,f_{j,i}\in C$, such that $f_i:=\sum_{l=1}^{j}f_{l,i}=\infty$ on $B_\epsilon(x_i)$.\\
		The Lipschitz regularization $f_{i,r}:=\reg_rf_i$ belongs to $\Conv(V,\R)$ for all $r>0$ small enough. Let $h_i\in\Conv(V,\R)$ be a convex function such that $h_i+\phi_i\in\Conv(V,\R)$. Then $\tilde{h}_{i,r}:=f_{i,r}+h_i\in C$ satisfies $\tilde{h}_{i,r}+\phi_i\in C$ as well, so
		\begin{align*}
		&\GWVConv(\mu)(\phi_1\otimes\dots\phi_k)\\
		=&\sum\limits_{l=0}^k(-1)^{k-l}\frac{1}{(k-l)!l!}\sum_{\sigma\in S_k}\bar{\mu}\left(\tilde{h}_{\sigma(1),r}+\phi_{\sigma(1)},...,\tilde{h}_{\sigma(l),r}+\phi_{\sigma(l)},\tilde{h}_{\sigma(l+1),r},...,\tilde{h}_{\sigma(k),r}\right)
		\end{align*}
		for all $r>0$ sufficiently small. Of course, $\tilde{h}_{i,r}$ epi-converges to $f_i+h_i$ for $r\rightarrow0$ and $\tilde{h}_{i,r}+\phi_i$ epi-converges to $f_i+h_i+\phi_i=f_i+h_i$, as $f_i=\infty$ on the support of $\phi_i$. The joint continuity of $\bar{\mu}$ implies
		\begin{align*}
		\GWVConv(\mu)(\phi_1\otimes\dots\phi_k)=& \sum\limits_{l=0}^k(-1)^{k-l}\frac{1}{(k-l)!l!}\sum_{\sigma\in S_k}\bar{\mu}(f_{\sigma(1)}+h_{\sigma(1)},...,f_{\sigma(k)}+h_{\sigma(k)})\\
		=&\bar{\mu}(f_1+h_1,...,f_k+h_k)\sum\limits_{l=0}^k(-1)^{k-l}\binom{k}{l}=0.
		\end{align*}
		\\
		For the second statement, let $\mu\in\VConv_k(V)$ be a valuation with support in $U_C$. If $f\in C$ is any function, it is finite and thus continuous on $U_C$. In particular, it is bounded on a compact neighborhood $A\subset U_C$ of the support of $\mu$ by definition of $U_C$. Taking a smaller neighborhood $U$ of the support of $\mu$ with $\bar U\subset  \mathrm{int}A$, Proposition \ref{proposition_convex_functions_local_lipschitz_constants} implies that $f$ is Lipschitz continuous on $U$. In particular, any subgradient of $f$ on $U$ has norm bounded by the Lipschitz constant. Proposition \ref{proposition_properties_LIpschitz_regularization} iv. implies that there exists $r_0>0$ such that $\reg_rf=f$ on $U$ for all $0<r\le r_0$. By Proposition \ref{proposition_support_convex_valuation} this shows that $\mu(\reg_rf)$ does not depend on $0<r\le r_0$ and thus
		\begin{align*}
			\mu'(f):=\lim\limits_{r\rightarrow 0}\mu(\reg_rf) 
		\end{align*}
		defines an extension of $\mu$ to $C$. Due to Proposition  \ref{proposition_properties_LIpschitz_regularization} v., it is a valuation. We need to show that this extension is continuous. As the topology on $C$ is metrizable, we only need to show that $\mu'$ is sequentially continuous. Let $(f_j)_j$ be a sequence in $C$ epi-converging to $f\in C$. Then all functions are finite on $U_C$ and thus they converge uniformly on the compact set $A\subset U_C$. The estimate in Proposition \ref{proposition_convex_functions_local_lipschitz_constants} shows that $\{f_j \ : \ j\in\mathbb{N}\}\cup \{f\}$ is uniformly Lipschitz continuous on $U$, so Proposition \ref{proposition_properties_LIpschitz_regularization} iii. implies that there exists $r_0>0$ such that $\reg_rf_j=f_j$ and $\reg_rf=f$ on $U$ for all $0<r\le r_0$ independent of $j\in\mathbb{N}$. In particular,  using Proposition \ref{proposition_support_convex_valuation} we see that there exists $r_0>0$ such that $\mu(\reg_r f)$ and $\mu(\reg_r f_j)$ do not depend on $0<r\le r_0$. As $\reg_r f_j\rightarrow \reg_r f$ for $j\rightarrow\infty$ and all $r$ sufficiently small, we obtain 
		\begin{align*}
			\mu(\reg_rf)=\lim\limits_{j\rightarrow\infty}\mu(\reg_rf_j)\quad \text{for all $r$ sufficiently small}.
		\end{align*}
		However, $\mu(\reg_r f_j)$ and $\mu(\reg_r f)$ are constant in $r$ for $0<r\le r_0$ independent of $j\in\mathbb{N}$, so we conclude $\mu'(f)=\lim\limits_{j\rightarrow\infty}\mu'(f_j)$.\\
		Obviously, the inclusion constructed this way is injective.\\
		
	\end{proof}
	Let us show that the both inclusions in Theorem \ref{maintheorem_support_valuations_non-finite_functions} are strict in general:\\
	Define $\mu(f):=f(0)+f(2)-2f(1)$ for $f\in\Conv(\R,\R)$. It is easy to see that $\mu$ is a dually epi-translation invariant valuation with support contained in $\{0,1,2\}$.\\
	For the first inclusion, let $C$ be the regular cone generated by $\Conv(\R,\R)$ and the \emph{convex indicator functions} $I^\infty_{[-\frac{1}{n},\infty)}$ for all $n\in\mathbb{N}$, where the convex indicator function of a closed convex set $K\subset \R$ is defined by
		\begin{align*}
		I^\infty_K(x):=\begin{cases}
		0 & x\in K,\\
		\infty & x\notin K.
		\end{cases}
	\end{align*}
	Then $U_C=(0,\infty)$ and any $f\in C$ satisfies $\supp\mu\subset \mathrm{int}\dom(f)$. Now let $(f_j)_j$ be a sequence in $C$ epi-converging to $f$. Due to Proposition \ref{proposition_convergence_finite_convex_functions}, the sequence converges locally uniformly on the interior of $\dom f$, so in particular on $\{0,1,2\}$, i.e. $\mu(f_j)=f_j(0)+f_j(2)-2f_j(1)\rightarrow f(0)+f(2)-2f(1)$. Thus we can extend $\mu$ continuously to $C$ by setting $\mu(f):=f(0)+f(2)-2f(1)$ for $f\in C$.\\
	
	For the second inclusion, let $C\subset\Conv(\R)$ be the regular cone generated by $\Conv(\R,\R)$ and the convex indicator $I^\infty_{[0,\infty)}$. Consider the sequence $(f_j)_j$ in $C$ given by
	\begin{align*}
	f_j(x)=j^2\max(\frac{1}{j}-x,0)= \begin{cases}
	j-j^2x & x\le\frac{1}{j},\\
	0 & x>\frac{1}{j}.
	\end{cases}
	\end{align*}
	Using Proposition \ref{proposition_convergence_finite_convex_functions} again, we see that $(f_j)_j$ epi-converges to $I^\infty_{[0,\infty)}$, but $\mu(f_j)=j$ for all $j\in\mathbb{N}$, so $\mu$ does not extend to $C$ by continuity.\\
	
	The restrictions on the support apply in particular to cones that are invariant under large subgroups of the affine group.
	\begin{corollary}
		\label{corollary_existence_non_trivial_valuations}
		Let $C\subset\Conv(V)$ be a regular cone containing $\Conv(V,\R)$ such that $C$ is invariant with respect to translations or $\mathrm{SL}(V)$ for $\dim V\ge 2$. If $C$ contains a non-finite convex function, then any dually epi-translation invariant valuation is constant.
	\end{corollary}
	\begin{proof}
		In both cases $B_C$ is either empty or contains only the origin. But for any $1\le k\le n$ there are no non-trivial valuations with this support due to Proposition \ref{proposition_no_concentration_of_support_in_1_point}. Thus the only valuations are the constant valuations.
	\end{proof}

	Let us see that in special cases we have an equality for the first inclusion in Theorem \ref{maintheorem_support_valuations_non-finite_functions}.
	\begin{proposition}
		\label{proposition_support_valuations_on_open_set}
		Let $U\subset V$ be an open convex set, $C_U:=\{f\in\Conv(V) \ : \ f|_U<\infty\}$, and $F$ a locally convex vector space. Then the support of any valuation $\mu\in \VConv(C_U;V,F)$ is contained in $U$.
	\end{proposition}
	\begin{proof}
		This is trivial for $U=V$, thus let us assume $U\ne V$. Due to Theorem \ref{maintheorem_support_valuations_non-finite_functions}, it is enough to show that the support of any valuation $\mu\in\VConv(C_U;V,F)$ does not contain any point $x_0\in\partial U$. By considering $\lambda\circ\mu$ for all $\lambda\in F'$, it is also sufficient to consider real-valued valuations. Let us assume that $\mu\in\VConv(C_U;V,F)$ is $k$-homogeneous for $1\le k\le n$ and that $x_0\in\supp\mu\cap\partial U$. Let us identify $V\cong\R^n$. By taking a supporting hyperplane through $x_0$ and using translations as well as rotations, we can assume that $x_0=0$ and that $\bar{U}\subset [0,\infty)\times\R^{n-1}$.\\
		As $0\in\supp\mu$, we can choose functions $\phi_1^j,...,\phi_k^j\in C^\infty_c(\R^n)$ with $\supp\phi_i^j\subset U_{\frac{1}{j}}(0)$ such that
		\begin{align*}
			\GWVConv(\mu)(\phi_1^j\otimes...\otimes\phi_k^j)=1  \quad\forall j\in\mathbb{N}.
		\end{align*}
		Consider the function $f_j\in C_U$ given by
		\begin{align*}
			h_j(x)=\begin{cases}
				\infty & x\in (-\infty,0)\times\mathbb{R}^{n-1},\\
				\max\left(\frac{(x_1-(j+\frac{1}{j}))^2}{2}+\sum\limits_{i=2}^{n}\frac{x_i^2}{2}, \frac{j^2}{2}\right)-\frac{j^2}{2} & x\in [0,\infty)\times\mathbb{R}^{n-1}.
			\end{cases}
		\end{align*}	
		Then $h_j\equiv 0$ on $B_{j}(j+\frac{1}{j},0,\dots,0)$. Setting $x_j:=(j+\frac{1}{j},0,\dots,0)$, we see that $x\in B_{j}(x_j)$ implies 
		\begin{align*}
			|x-x_{j+1}|\le |x-x_j|+|x_j-x_{j+1}|\le j+\frac{1}{j}-\frac{1}{j+1}\le j+1,
		\end{align*}
		so $B_j(x_j)\subset B_{j+1}(x_{j+1})$. If $y=(y_1,...,y_n)\in (0,\infty)\times\mathbb{R}^{n-1}$ is given, \begin{align*}
			|y-x_j|^2-j^2=-2\left(j+\frac{1}{j}\right)y_1+\frac{1}{j^2}+2+\sum_{i=1}^{n}y_i^2\rightarrow-\infty \quad \text{for }j\rightarrow\infty,
		\end{align*}
		so we see that $\bigcup_{j\in\mathbb{N}}B_j(x_j)=(0,\infty)\times\R^{n-1}$. In particular, the sequence $(h_j)_j$ converges pointwise to $h:=I^\infty_{[0, \infty)\times\mathbb{R}^{n-1}}$ for all $x\notin \{0\}\times \R^{n-1}$. By Proposition \ref{proposition_convergence_finite_convex_functions} this implies that $(h_j)_j$ epi-converges to $h$.\\
		Now set $c_j:=\max_{i=1,...,k}\|\phi_i^j\|_{C^2(V)}$ and define $f_i^j:= c_jh_j+\phi_i^j$. Then $f_i^j\in C_U$ for all $1\le i\le k$, $j\in\mathbb{N}$. By construction $\lim_{j\rightarrow\infty}f_i^j(x)=I^\infty_{[0, \infty)\times\mathbb{R}^{n-1}}(x)=h(x)$ for $x\notin \{0\}\times\mathbb{R}^{n-1}$, so Proposition \ref{proposition_convergence_finite_convex_functions} shows that $(f_i^j)_j$ epi-converges to $h$ for $j\rightarrow\infty$.
		Using the definition of the Goodey-Weil embedding and the joint continuity of the polarization $\bar{\mu}$, we obtain the contradiction
		\begin{align*}
			 1=&\lim\limits_{j\rightarrow\infty}\GWVConv(\mu)(\phi_1^j\otimes...\otimes\phi_k^j)\\
			 =&\lim\limits_{j\rightarrow\infty}\sum\limits_{l=0}^k(-1)^{k-l}\frac{1}{(k-l)!l!}\sum_{\sigma\in S_k}\bar{\mu}\left(f^j_{\sigma(1)},...,f^j_{\sigma(l)},h_{\sigma(l+1)},\dots,h_{\sigma(k)}\right)\\
			=&\sum\limits_{l=0}^k(-1)^{k-l}\frac{1}{(k-l)!l!}\sum_{\sigma\in S_k}\bar{\mu}(h[l],h[k-l])\\
			=& (-1)^k\mu(h)\sum\limits_{l=0}^k(-1)^{l}\frac{k!}{(k-l)!l!}=0.
		\end{align*}
		Thus $0\notin\supp \mu$.
	\end{proof}
	
\subsection{Valuations on convex functions on open convex sets}
	\label{section_valuations_on_open_sets}
	For an open and convex subset  $U\subset V$ let us denote by $\Conv(U,\R)$ the space of all convex functions $f:U\rightarrow\R$. Equipped with the topology of uniform convergence on compact subset of $U$, $\Conv(U,\R)$ becomes a metrizable topological space.
	\begin{lemma}
		For $f\in\Conv(U,\R)$, define $\tilde{f}$ by
		\begin{align*}
			\tilde{f}(x_0)=\begin{cases}
				f(x_0)  & x_0\in U,\\
				\liminf\limits_{x\rightarrow x_0,x\in U}f(x) & x_0\in\partial U,\\
				\infty & x_0\in V\setminus\bar{U}.
			\end{cases}
		\end{align*}
		Then $\tilde{f}\in\Conv(V)$.
	\end{lemma}
	\begin{proof}
		Observe that $\tilde{f}(x_0)=\liminf_{x\rightarrow x_0,x\in U}f(x)$ for all $x\in\bar{U}$, as $f$ is continuous on $U$.\\
		Obviously, $\tilde{f}$ is lower semi-continuous. We need to show that $\tilde{f}>-\infty$ and that $\tilde{f}$ is convex.\\
		Let $x\in\partial U$ be any point, $(x_j)_j$ a sequence in $U$ converging to $x$ such that $\lim_{j\rightarrow\infty}f(x_j)=\tilde{f}(x)$. For $y\in U$ and $\lambda\in (0,1)$ the convexity of $f$ implies
		\begin{align*}
			f(\lambda y+(1-\lambda)x_j)\le \lambda f(y)+(1-\lambda)f(x_j).
		\end{align*}
		As $U$ is open, $\lambda y+(1-\lambda)x_j\rightarrow \lambda y+(1-\lambda)x$ in $U$ for all $\lambda\in (0,1)$, so the continuity of $f$ implies
		\begin{align*}
			f(\lambda y+(1-\lambda)x)\le \lambda f(y)+(1-\lambda)\tilde{f}(x).
		\end{align*}
		In particular, $\tilde{f}(x)>-\infty$. In addition, we see that $\tilde{f}$ is convex along line segments $[x,y]$, where $x\in\partial U$ and $y\in U$. To see that $\tilde{f}$ is convex, the only non-trivial case remaining is a line segment $[x,y]$ where $x,y\in\partial U$. Take a sequence $(y_j)_j$ in $U$ converging to $y$ such that $\lim\limits_{j\rightarrow\infty}f(y_j)=\tilde{f}(y)$. Using the inequality above we see that for $\lambda\in (0,1)$
		\begin{align*}
			f(\lambda y_j+(1-\lambda)x)\le \lambda f(y_j)+(1-\lambda)\tilde{f}(x).
		\end{align*}
		Now $\lambda y_j+(1-\lambda)x\in U$ defines a sequence converging to $\lambda y+(1-\lambda)x\in\bar{U}$. Thus taking limits and using the remark we obtain 
		\begin{align*}
			\tilde{f}(\lambda y+(1-\lambda)x)\le\liminf\limits_{j\rightarrow\infty}f(\lambda y_j+(1-\lambda)x)\le \lambda \tilde{f}(y)+(1-\lambda)\tilde{f}(x).
		\end{align*}
	\end{proof}
	\begin{proposition}
		\label{proposition_continuity_extension_from_open_convex_set}
		The extension $f\mapsto \tilde{f}$ defines a continuous, injective map $i_U:\Conv(U,\R)\rightarrow C_U:=\{f\in\Conv(V) \ : \ f|_U<\infty\}$. The inverse map is given by restricting the map
		\begin{align*}
			\res:C_U&\rightarrow \Conv(U,\R)\\
			f&\mapsto f|_U 
		\end{align*}
		to the image of $\Conv(U,\R)$ in $C_U$ and is also continuous. 
		In addition, $i_U$ and $\res$ are compatible with the formation of pointwise maximum and minimum of two convex functions.
	\end{proposition}
	\begin{proof}
		It is clear that $i_U$ is injective. To see that it is continuous, it is enough to show that it is sequentially continuous, as both spaces are metrizable.\\
		Let $(f_j)_j\subset\Conv(U,\R)$ be a sequence converging to $f\in\Conv(U,\R)$. Then $(\tilde{f}_j)_j$ converges pointwise on the dense subset $V\setminus\partial U$ to $\tilde{f}$, so the claim follows from Proposition \ref{proposition_convergence_finite_convex_functions}. Of course, the restriction map defines the inverse to this extension procedure. The continuity follows again from Proposition \ref{proposition_convergence_finite_convex_functions}.\\
		Obviously, the restriction map is compatible with the formation of the pointwise maximum and minimum. If $f,h\in\Conv(U,\R)$, then $i_U(\max(f,h))=\max (i_U(f),i_U(h))$ on $V\setminus \partial U$. Thus $i_U(\max(f,h))(x_0)\le\max(i_U(f),i_U(h))(x_0)$ for $x_0\in\partial U$ by definition of $i_U$. For the reverse inequality, take a sequence $(x_j)_j$ in $U$ converging to $x_0$ such that
		\begin{align*}
			\lim\limits_{j\rightarrow\infty}\max(f(x_j),h(x_j))=i_U(\max(f,h))(x_0).
		\end{align*}
		As $i_U(f)$ and $i_U(h)$ are lower semi-continuous, given $\epsilon>0$ there exists $N\in\mathbb{N}$ such that $i_U(f)(x_0)\le f(x_j)+\epsilon$ and $i_U(h)(x_0)\le h(x_j)+\epsilon$ for all $j\ge N$ and thus
		\begin{align*}
			\max(i_U(f),i_U(h))(x_0)\le\max(f(x_j),h(x_j))+\epsilon \quad\forall j\ge N.
		\end{align*}
		Thus $\max(i_U(f),i_U(h))(x_0)\le i_U(\max(f,h))(x_0)$. The same argument can be applied to the minimum.
	\end{proof}
	\begin{definition}
		Let $U$ be an open convex subset, $F$ a real locally convex vector space. We will denote the space of all continuous valuations $\mu:\Conv(U,\R)\rightarrow F$ that are dually epi-translation invariant by $\VConv(U,F)$.
	\end{definition}
	As usual, we equip $\VConv(U,F)$ with the topology of uniform convergence on compact subsets, which is generated by the semi-norms $\|\mu\|_{F;K}:=\sup_{f\in K}|\mu(f)|_F$ for all continuous semi-norms $|\cdot|_F$ of $F$ and compact subsets $K\subset \Conv(U,\R)$.\\
	Let $C_U:=\{f\in\Conv(V) \ : \ f|_U<\infty\}$ be the cone of convex functions that are finite on $U$. Using Proposition \ref{proposition_continuity_extension_from_open_convex_set}, we can consider the map
	\begin{align*}
		\res^*:\VConv(U,F)&\rightarrow \VConv(C_U;V,F)\\
		\mu&\mapsto [f\mapsto \mu(f|_U)].
	\end{align*}
	\begin{lemma}
		\label{lemma_properties_inclusion_open_subset}
		$\res^*:\VConv(U,F)\rightarrow \VConv(C_U;V,F)$ is injective and continuous.
	\end{lemma}
	\begin{proof}
		Assume that $\res^*(\mu)=0$. Let $f\in\Conv(U,\R)$. The Lipschitz regularization $\reg_r \tilde{f}$ belongs to $\Conv(V,\R)$ for $r>0$ small enough, Proposition \ref{proposition_properties_LIpschitz_regularization} together with Proposition \ref{proposition_continuity_extension_from_open_convex_set} implies that $\mu([\reg_r\tilde{f}]|_U)$ converges to $\mu(f)$. However $\mu([\reg_r\tilde{f}]|_U)=0$, so $\mu(f)=0$. As this holds for arbitrary $f\in\Conv(U,\R)$, $\mu=0$.\\
		To see that the map is continuous, let $K\subset C_U$ be a compact subset. As the restriction $\res:C_U\rightarrow\Conv(U,\R)$ is continuous due to Proposition \ref{proposition_continuity_extension_from_open_convex_set}, $\res(K)\subset\Conv(U,\R)$ is compact, and 
		\begin{align*}
			\|\res^*\mu\|_{F;K}=\|\mu\|_{F;\res(K)}.
		\end{align*}
		Thus $\res^*$ is continuous.
	\end{proof}
	In addition to $\res^*$, we can also consider
	\begin{align*}
		i_{U}^*:\VConv(C_U;V,F)&\rightarrow\VConv(U,F)\\
		\mu&\mapsto [f\mapsto \mu(i_U(f))].
	\end{align*}
	Using the same argument as in Lemma \ref{lemma_properties_inclusion_open_subset}, we see that this is well defined and continuous.
	\begin{proposition}
		\label{proposition_isormorphism_restriction}
		Let $F$ be a locally convex vector space admitting a continuous norm. Then $i_{U}^*$ and $\res^*$ are topological isomorphisms and inverse to each other.
	\end{proposition}
	\begin{proof}
		It is easy to see that $i_{U}^*\circ \res^*=Id_{\VConv(U,F)}$, so $i_U^*$ is surjective. Let us show that $i_U^*$ is injective. Assume that $\mu\in\VConv(C_U;V,F)$ satisfies $\mu(i_U(f))=0$ for all $f\in\Conv(U,\R)$. Due to Proposition \ref{proposition_support_valuations_on_open_set}, the support of $\mu$ is compactly contained in $U$. Given $h\in\Conv(V,\R)$, the function $i_U(h|_U)$ coincides with $h$ on $U$, i.e. they coincide on a neighborhood of the support of $\mu$. Proposition \ref{proposition_support_convex_valuation} implies $\mu(h)=\mu(i_U(h|_U))=0$. Thus $\mu$ vanishes on the dense subset $\Conv(V,\R)\subset C_U$, i.e. $\mu=0$.\\
		We obtain $(i_U^*)^{-1}=\res^*$, which is continuous. The same applies to $(\res^*)^{-1}=i_U$.
	\end{proof}
	\begin{proof}[Proof of Theorem \ref{maintheorem_isomorphism_cone_open_subset}]
		This is just a reformulation of Proposition \ref{proposition_isormorphism_restriction}.
	\end{proof}

\addcontentsline{toc}{section}{References}
\bibliography{literature}

\begin{thebibliography}{10}

\bibitem{Alesker:McMullenconjecture}
Semyon Alesker.
\newblock {On {P}. {McMullen}'s conjecture on translation invariant
  valuations}.
\newblock {\em Advances in Mathematics}, 155(2):239--263, 2000.

\bibitem{Alesker:IrreducibilityThm}
Semyon Alesker.
\newblock {Description of translation invariant valuations on convex sets with
  solution of {P}. {McMullen}'s conjecture}.
\newblock {\em Geometric \& Functional Analysis GAFA}, 11(2):244--272, 2001.

\bibitem{Alesker:valuations_convex_functions_Monge-Ampere}
Semyon Alesker.
\newblock {Valuations on convex functions and convex sets and
  {Monge}--{Amp{\`e}re} operators}.
\newblock {\em Advances in Geometry}, 19(3):313--322, 2019.

\bibitem{Bobkov_Colesanti:quermassintegrals_quasi-concave_functions}
Sergey~G Bobkov, Andrea Colesanti, and Ilaria Fragala.
\newblock {Quermassintegrals of quasi-concave functions and generalized
  {P}r{\'e}kopa--{L}eindler inequalities}.
\newblock {\em Manuscripta Mathematica}, 143(1-2):131--169, 2014.

\bibitem{Cavallina_Colesanti:monotone_valuations_convex_functions}
Lorenzo Cavallina and Andrea Colesanti.
\newblock {Monotone valuations on the space of convex functions}.
\newblock {\em Analysis and Geometry in Metric Spaces}, 3(1), 2015.

\bibitem{Clarke:optimization_non_smooth_analysis}
Frank~H. Clarke.
\newblock {\em Optimization and nonsmooth analysis}, volume~5 of {\em Classics
  in applied mathematics}.
\newblock Society for Industrial and Applied Mathematics, Philadelphia, 1990.

\bibitem{Colesanti_Lombardi:valuations_quasi-concave}
Andrea Colesanti and Nico Lombardi.
\newblock Valuations on the space of quasi-concave functions.
\newblock In B.~Klartag and E.~Milman, editors, {\em Geometric Aspects of
  Functional Analysis. Lecture Notes in Mathematics}, volume 2169, pages
  71--105. Springer, 2017.

\bibitem{Colesanti_Lombardi_Parapatits:translation_invariant_valuations_quasi-concave}
Andrea Colesanti, Nico Lombardi, and Lukas Parapatits.
\newblock {Translation invariant valuations on quasi-concave functions}.
\newblock {\em Studia Mathematica}, 243:79--99, 2018.

\bibitem{Colesanti_Ludwig_Mussnig:minkowski}
Andrea Colesanti, Monika Ludwig, and Fabian Mussnig.
\newblock {{Minkowski} valuations on convex functions}.
\newblock {\em Calculus of variations and partial differential equations},
  56(6):162, 2017.

\bibitem{Colesanti_Ludwig_Mussnig:Valuations_convex_functions}
Andrea Colesanti, Monika Ludwig, and Fabian Mussnig.
\newblock {Valuations on convex functions}.
\newblock {\em International Mathematics Research Notices}, 2017.

\bibitem{Colesanti_Ludwig_Mussnig:homogeneous_decomposition}
Andrea Colesanti, Monika Ludwig, and Fabian Mussnig.
\newblock {A homogeneous decomposition theorem for valuations on convex
  functions}.
\newblock {\em Journal of Functional Analysis}, 279(5):108573, 2020.

\bibitem{Colesanti_Ludwig_Mussnig:Hessian_valuations}
Andrea Colesanti, Monika Ludwig, and Fabian Mussnig.
\newblock {Hessian valuations}.
\newblock {\em Indiana University Mathematics Journal}, 69(4), 2020.

\bibitem{Gask:Schwartz_kernel}
H.~Gask.
\newblock A proof of {S}chwartz's kernel theorem.
\newblock {\em Mathematica Scandinavica}, 8(2):327--332, 1961.

\bibitem{Goodey_Weil:Distributions_and_valuations}
Paul Goodey and Wolfgang Weil.
\newblock {Distributions and valuations}.
\newblock {\em Proceedings of the London Mathematical Society}, 3(3):504--516,
  1984.

\bibitem{Kone:valuations_orlicz}
Hassane Kone.
\newblock {Valuations on {Orlicz} spaces and {$L^\phi$}-star sets}.
\newblock {\em Advances in Applied Mathematics}, 52:82--98, 2014.

\bibitem{Ludwig:Fisher_information_valuations}
Monika Ludwig.
\newblock {{Fisher} information and matrix-valued valuations}.
\newblock {\em Advances in Mathematics}, 226(3):2700--2711, 2011.

\bibitem{Ludwig:valuations_sobolev}
Monika Ludwig.
\newblock {Valuations on {Sobolev} spaces}.
\newblock {\em American Journal of Mathematics}, 134(3):827--842, 2012.

\bibitem{Ludwig:covariance_matrices_valuations}
Monika Ludwig.
\newblock {Covariance matrices and valuations}.
\newblock {\em Advances in Applied Mathematics}, 51(3):359--366, 2013.

\bibitem{Ma:valuations_sobolev}
Dan Ma.
\newblock {Real-valued valuations on {Sobolev} spaces}.
\newblock {\em Science China Mathematics}, 59(5):921--934, 2016.

\bibitem{McMullen:Euler_type_McMullen_decomposition}
Peter McMullen.
\newblock {Valuations and {Euler}-Type Relations on Certain Classes of Convex
  Polytopes}.
\newblock {\em Proceedings of the London Mathematical Society}, 3(1):113--135,
  1977.

\bibitem{Mussnig:SLn_invariant_super_coercive}
Fabian Mussnig.
\newblock {{$SL(n)$} invariant valuations on super-coercive convex functions}.
\newblock {\em Canadian Journal of Mathematics (in press)}.
\newblock arXiv:1903.04225.

\bibitem{Mussnig:volume_polar_volume_euler}
Fabian Mussnig.
\newblock {Volume, polar volume and {Euler} characteristic for convex
  functions}.
\newblock {\em Advances in Mathematics}, 344:340--373, 2019.

\bibitem{Ober:Minkowski_valuations_on_Lq_spaces}
Michael Ober.
\newblock {{$L_p$}-{Minkowski} valuations on {$L^q$}-spaces}.
\newblock {\em Journal of Mathematical Analysis and Applications},
  414(1):68--87, 2014.

\bibitem{Rockafellar:Convex_Analysis}
R.~Tyrrell Rockafellar.
\newblock {\em Convex analysis}, volume~28 of {\em Princeton mathematical
  series}.
\newblock Princeton University Press, Princeton, NJ, 1970.

\bibitem{Rockafellar_Wets:Variational_analysis}
R.~Tyrrell Rockafellar and Roger J.-B. Wets.
\newblock {\em Variational analysis}, volume 317 of {\em Die Grundlehren der
  mathematischen Wissenschaften in Einzeldarstellungen}.
\newblock Springer, Berlin, 2009.

\bibitem{Schneider:convex_bodies_Brunn-Minkowski}
Rolf Schneider.
\newblock {\em Convex bodies: the {Brunn}--{Minkowski} theory}, volume 151 of
  {\em Encyclopedia of Mathematics and its Applications}.
\newblock Cambridge University Press, Cambridge, 2 edition, 2014.

\bibitem{Tsang:valuations_Lp}
Andy Tsang.
\newblock {Valuations on {$L^p$}-spaces}.
\newblock {\em International Mathematics Research Notices},
  2010(20):3993--4023, 2010.

\bibitem{Tsang:minkowski_val_Lp}
Andy Tsang.
\newblock {{Minkowski} valuations on {$L^p$}-spaces}.
\newblock {\em Transactions of the American Mathematical Society}, pages
  6159--6186, 2012.

\bibitem{Wang:semivaluations_bounded_variations}
Tuo Wang.
\newblock {Semi-Valuations on {$BV(\mathbb{R}^n)$}}.
\newblock {\em Indiana University Mathematics Journal}, pages 1447--1465, 2014.

\bibitem{Yan:Extension_convex_function}
Min Yan.
\newblock {Extension of Convex Function}.
\newblock {\em Journal of Convex Analysis}, 21(4):965--987, 2012.

\end{thebibliography}

\Addresses

\end{document}